\documentclass[11 pt]{amsart}
\usepackage{amssymb,tikz}
\usepackage{caption}
\usepackage{subcaption}
\usepackage{amsmath}
\usepackage{amsthm}
\usepackage{amsbsy}
\usepackage{psfrag}
\usepackage{pstricks}
\usepackage{graphics}
\usepackage{graphicx}
\usepackage[english]{babel}
\usepackage[autostyle]{csquotes}
\usepackage{epstopdf,bm}
\usepackage{enumitem}  
\def\b#1{\boldsymbol{#1}}
\usepackage{enumitem}
\usepackage[autostyle]{csquotes}
\newcounter{example}[section]

\usepackage{array}
\newcolumntype{P}[1]{>{\centering\arraybackslash}p{#1}}
\theoremstyle{plain}
\newtheorem{theorem}{Theorem}[section]
\newtheorem{lemma}[theorem]{Lemma}

\newtheorem{Problem}[theorem]{Problem}

\theoremstyle{remark}
\newtheorem{remark}[theorem]{Remark}

\begin{document}
	\allowdisplaybreaks[4]
	\numberwithin{figure}{section}
	\numberwithin{table}{section}
	\numberwithin{equation}{section}
	%
	\title[$L^{\infty}-$Quadratic DG FEM for the Unilateral contact problem]
	{Pointwise A posteriori error control of quadratic Discontinuous Galerkin Methods for the unilateral contact problem}
	\author{Rohit Khandelwal}
	\email{rohitkhandelwal004@gmail.com}
\address{Department of Mathematics, Government Science College Chikhli, Navsari, Gujarat - 396521}
	\author{Kamana Porwal}
	\email{kamana@maths.iitd.ac.in}
	\address{Department of Mathematics, Indian Institute of Technology Delhi, New Delhi - 110016}
	\author{Tanvi Wadhawan}
	\email{maz188452@maths.iitd.ac.in}
	\address{Department of Mathematics, Indian Institute of Technology Delhi, New Delhi - 110016}

	\date{}
	\begin{abstract}
    An a posteriori error bound for the  pointwise error of the quadratic discontinuous Galerkin method for the unilateral contact problem on polygonal domain is presented. The pointwise a posteriori error analysis is based on the direct use of a priori estimates of the Green's matrix for the divergence type operators and the suitable construction of the discrete contact force density $\b{\sigma}_h$ and barrier functions for the continuous solution. Several numerical experiments (in two dimension) are presented to illustrate the reliability and efficiency properties of the proposed aposteriori error estimator.
	\end{abstract}
	\keywords{Finite element method,  pointwise a posteriori error estimates, variational inequalities, Signorini problem, Continuous contact force density, supremum norm}
	\subjclass{65N30, 65N15}
	\maketitle
	\allowdisplaybreaks
	\def\R{\mathbb{R}}
	\def\cA{\mathcal{A}}
	\def\cK{\mathcal{K}}
	\def\cN{\mathcal{N}}
	\def\p{\partial}
	\def\O{\Omega}
	\def\bbP{\mathbb{P}}
	\def\cV{\mathcal{V}}
	\def\cM{\mathcal{M}}
	\def\cT{\mathcal{T}}
	\def\cE{\mathcal{E}}
	\def\cF{\mathcal{F}}
	\def \cW{\mathcal{W}}
	\def \cJ{\mathcal{J}}
	\def \cV{\mathcal{V}}
	\def\bF{\mathbb{F}}
	\def \cW{\mathcal{F}}
	\def\bC{\mathbb{C}}
	\def\bN{\mathbb{N}}
	\def\ssT{{\scriptscriptstyle T}}
	\def\HT{{H^2(\O,\cT_h)}}
	\def\mean#1{\left\{\hskip -5pt\left\{#1\right\}\hskip -5pt\right\}}
	\def\jump#1{\left[\hskip -3.5pt\left[#1\right]\hskip -3.5pt\right]}
	\def\smean#1{\{\hskip -3pt\{#1\}\hskip -3pt\}}
	\def\sjump#1{[\hskip -1.5pt[#1]\hskip -1.5pt]}
	\def\jumptwo{\jump{\frac{\p^2 u_h}{\p n^2}}}
	
	\section{Introduction}
\par
\noindent
A popular and efficient technique for approximating solutions of partial differential equations (PDE's) is adaptive finite element method (AFEM).  The crucial step in designing adaptive finite element method is deriving a posteriori error estimators.  These error estimators establish a link between the error and variables that can be calculated using the discrete solution and available data.  We refer to the books \cite{verfurth1996review,ainsworth1997posteriori} for an overview of these methods and their development.
\\ 

\noindent
 Several problems in contact mechanics are modeled as variational inequalities which play an essential role in solving class of various non-linear boundary value problems arising in the physical fields.  In this article,  we consider the scenario of the Signorini problem (frictionless unilateral contact problem in linear elasticity) which is basically studied as a prototype for elliptic variational inequality (EVI) of the first kind.
\vspace{1.5mm}
\par
\noindent
\underline{\textbf{The Unilateral Contact Problem:}}
\vspace{1mm}
\par
\noindent	
Consider a bounded polygonal domain $\Omega \subseteq \mathbb{R}^2$  with boundary $\partial \Omega $,  which represents the reference configuration of an elastic body.  The boundary of $\Omega $ consists of three non-overlapping parts $\Gamma_N$ (the Neumann boundary), $\Gamma_D$ (the Dirichlet boundary) and the contact
boundary $\Gamma_C$.  We assume that
$meas(\Gamma_D) > 0$ and ${\bar\Gamma_C} \cap {\bar\Gamma_D} = \emptyset$ to avoid dealing with the space $H^{\frac{1}{2}}_{00}(\Gamma_C)$.  Additionally we assume that $\Gamma_C$ is a straight line segment.  The outward unit normal vector on $\partial \Omega$ is denoted by $\b{n}$. Initially,  the body is in contact with a rigid foundation on the potential contact boundary.  For the ease of analysis, the body $\Omega$ is clamped on $\Gamma_D$.  Let $\chi : \Gamma_C \rightarrow \mathbb{R}$ be the non-negative gap function s.t. $\chi \in L^{\infty}(\Gamma_C)$.  Further $\Omega$ is subjected to the volume forces $\b{f} \in [L^{\infty}(\Omega)]^2$ and the surface loads $\b{\pi} \in [L^{\infty}(\Gamma_N)]^2$ acts on Neumann boundary.  The unilateral contact problem in linear elasticity under consideration reads: find the displacement vector $\b{u}: \Omega \rightarrow \mathbb{R}^2$ verifying the equations \eqref{problem1}--\eqref{problem3}
\begin{equation} \label{problem1}
\begin{cases}
-\mbox{$div$}~\bm{\Xi}\b{(u)}  = \b{f} \hspace{0.45cm}\mbox{in}~ \Omega, & \\ \hspace{0.7cm} \bm{\Xi(u)n} = \mathbf{\bm{\pi}} \hspace{0.4cm}\mbox{on}~ \Gamma_N, & \\ \hspace{1.6cm}  \b{u} = \bm{0 }~ \quad \mbox{on}~ \Gamma_D, 
\end{cases}
\end{equation}
\par
\noindent
where for any $\bm{A}= (a_{ij}) \in \mathbb{R}^{2 \times 2}$,  the divergence of $\bm{A}$ is defined by
$(div(\bm{A}))_i := \sum\limits_{j=1}^{2} \frac{\partial}{\partial x_{j}} (a_{ij}), ~i=1,2$ and the linearized stress tensor $\bm{\Xi}: \Omega \rightarrow \mathbb{R}^{2 \times 2}$ is given by
\begin{align*}
\bm{\Xi}(\bm{w}):= \mathcal{C}\bm{\epsilon(w)}.
\end{align*}
Here,
\begin{enumerate}
	\item[a)] $\mathcal{C}$ is the fourth order bounded, symmetric and positive definite elasticity tensor.
	\item[b)] $\bm{\epsilon(w)}= \frac{1}{2} (\bm{\nabla w}^T + \bm{\nabla w}) $ is the linearized strain tensor.
\end{enumerate}
\par
\noindent
Due to the homogeneity and isotropy of the elastic body being studied,  it follows that Hooke's law is applicable. Thus the  stress tensor can be represented as 
\begin{align*}
\mathcal{C}\b{\varepsilon}(\b{u}):= 2\mu\b{\varepsilon}(\b{u})+\kappa(\text{tr} \b{\varepsilon}(\b{u}))\textbf{I},
\end{align*}
where,  $\mu>0$ and $\kappa>0$ denote the Lam$\acute{e}$ parameters and $ \textbf{I}$ is an identity matrix of order 2.  
\vspace{0.2 cm}
\par
\noindent
For any displacement field $\b{w}$, we adopt the notation $w_n=\b{w}\cdot \b{n}$ and $\b{w}_t=\b{w}-w_n\b{n}$ respectively,  as its normal and tangential component on the boundary.  Similarly,  for a tensor-valued function ${\b{\varphi}
(\b{w})}$ the normal and tangential component are defined as ${{\varphi}_n(\b{w})}=(\b{{\b{\varphi}(\b{w})n}})\cdot \b{n}$ and  $\b{\varphi}_t(\b{w})=\b{\varphi}(\b{w})\b{n}-{\varphi}_n(\b{w}) \b{n}$,  respectively.   Further,  we have the following decomposition formula
\begin{align*}
  ({\b{\varphi}(\b{w})}\b{n})\cdot\b{w}= {\varphi_n(\b{w})}w_n+ {\b{\varphi}_t(\b{w})}\cdot\b{w}_t.
\end{align*}
The conditions describing the unilateral contact without friction on $\Gamma_C$ are as follows:
\begin{equation} \label{problem3}
\begin{cases}
\hspace{1.9cm} u_n \leq \chi, & \\  \hspace{1.3cm} {\Xi}_n(\b{u}) \leq 0, & \\ (u_{n}-\chi) {\Xi}_n(\b{u}) =0,  & \\ \hspace{1.4cm} \bm{{\Xi}}_{\b{t}}(\b{u}) =0. 
\end{cases}
\end{equation}

A posteriori error estimates in the supremum norm are of great importance in nonlinear problems,  particularly when the solution represents a physical quantity that requires accurate pointwise evaluation.  In the context of the linear elliptic problem,  the reliable and efficient a posteriori error estimates in the energy norm using continuous finite element method has been studied in \cite{verfurth1996review}.  Additionally,  in the article \cite{karakashian2003posteriori},  the authors explore a novel a posteriori error estimates in the energy norm for the class of discontinuous Galerkin (DG) methods for the  second-order linear elliptic problem.
Furthermore,  the articles \cite{nochetto1995pointwise,demlow2012pointwise} analyze a posteriori error estimation techniques in the supremum norm for linear elliptic problems using conforming and DG finite element methods, respectively. In the article \cite{nochetto1995pointwise}, the focus lies on a posteriori error estimates in the supremum norm for the linear elliptic problem in two dimensions,  employing graded meshes to achieve optimal accuracy. Followed by that in  \cite{dari1999maximum}, the earlier approach has been extended to the three dimensional space.  We refer to the articles \cite{weiss2009posteriori,krause2015efficient,gudi2016posteriori,walloth2019reliable}  for the work on the a posteriori error analysis in the energy norm using conforming and DG methods for the Signorini problem.  The convergence analysis of conforming finite element method in the supremum norm for the variational inequalities is discussed in the article \cite{baiocchi1977estimations}.  The pointwise a posteriori error control of linear conforming finite element method for the one body contact problem is dealt in \cite{KP:2021:Signorini}.  
Further,  in the article \cite{KP:2022:Signorini},  pointwise a posteriori error analysis of linear DG methods for the unilateral contact problem has been addressed.  In articles \cite{KP:2022:QuadSignorini} and \cite{KP:2022:QuadLinfSignorini} the authors derived the reliable and efficient a posteriori error analysis using quadratic finite element method for the Signorini problem in the energy norm and supremum norm,  respectively.  Therein, the authors carried out the analysis by defining the continuous Lagrange multiplier as a functional on $\b{H}^{-1}(\O)$.  {Recently,  in \cite{KP:2022:QuadDGSignorini}, the authors provided a rigorous analysis for the DG discretization with quadratic polynomials to Signorini contact problem and perform a priori and a posteriori error analysis in the energy norm.  In contrast to \cite{KP:2022:QuadSignorini},  the authors in the article \cite{KP:2022:QuadDGSignorini} constructed the Lagrange multiplier as functional on $\b{H}^{\frac{1}{2}}(\Gamma_C)$.}
\vspace{0.2 cm}
\par
\noindent
In this article, we derive the reliable and efficient pointwise a posteriori error estimators of quadratic discontinuous Galerkin finite element method for the Signorini contact problem.  The analysis hinges on the introduction of upper and lower barriers of the continuous solution $\bm{u}$, the sign property of quasi-discrete contact force density and the known a priori bounds on the Green's matrix of the divergence type operator.  To the best of the knowledge of authors, 
the analysis developed involves novel residual type aposteriori error estimates of quadratic finite element method in the supremum norm for analyzing the unilateral contact problem. 
\vspace{0.2 cm}
\par
\noindent
The article is structured as follows: Section \ref{sec2} outlines the variational formulation of the unilateral contact problem along with introducing an auxiliary functional $\b{\sigma}$ defined on the space $\b{H}^{-\frac{1}{2}}(\Gamma_C)$ for the exact solution and complimentary conditions pertaining to the contact region.  In Section \ref{sec3} we provide the prerequisite notations  and discuss the discrete version of the continuous problem on a closed, convex,  non-empty subset of quadratic finite element space followed by that in Section \ref{sec4} we  introduce the discrete counterpart of continuous contact force density $\b{\sigma}_h$ on a suitable space and analyse its sign properties.  Later,  we construct the smoothing operator $\b{E}_h$ in order to deal with discontinuous functions.  Section \ref{sec5} provides a unified a posteriori error analysis for various DG methods.  Therein, the reliability and efficiency of a posteriori error estimators has been discussed using the bounds of existing a priori error estimates for the Green's matrix of the divergence type operator and barrier functions for the continuous solution. Finally,  Section \ref{sec6} presents two numerical experiments that confirm the theoretical results.  The convergence behavior of the error estimator as well as refined meshes are shown for two DG methods (SIPG and NIPG).  It is observed that mesh is refined more near the free boundary and thus the transition zone between actual contact and non-contact region is  captured well.
\section{Continuous Problem} \label{sec2}
\noindent
To commence this section,  we introduce the following spaces that shall be relevant in the further analysis.
\begin{itemize}
\item  $L^p(\O)$ refers to the space of Lebesgue measurable functions endowed with the norm
\begin{align*}
\|\psi\|_{L^p(\O)} :=\big(\int\limits_{\O} |\psi(x)|^p~dx \big)^{\frac{1}{p}}\quad\forall~\psi\in L^p(\O).
\end{align*}
\item The Sobolev space $W^{m,p}(\O)$ is the collection of $L^p(\O)$ functions such that the weak derivative upto order $m$ are also in $L^p(\O)$.  The norm $\|\cdot\|_{W^{m,p}(\O)}$ and semi norm $
|\cdot|_{W^{m,p}(\O)}$ on these spaces are defined by
\begin{align*}
\|{\psi} \|_{W^{m,p}(\O)} := \bigg(\underset{0 \leq |\alpha| \leq m}{\sum} \|\partial^{\alpha} {v}\|^p_{L^p(\O)}\bigg)^{\frac{1}{p}}
~\text{and}~   
| \psi |_{W^{m,p}(\O)} := \bigg(\underset{|\alpha| = m}{\sum} \|\partial^{\alpha} {\psi}\|^p_{L^p(\O)}\bigg)^{\frac{1}{p}},
\end{align*}
 respectively.  Here,  
 $\alpha= (\alpha_1, \alpha_2)$ represents the multi index in $\mathbb{N}^2$ and the symbol $\partial^{\alpha}$ refers to the partial derivative of $v$ defined as $\partial^{\alpha}v := \frac{\partial^{|\alpha|}v}{\partial x^{\alpha_1} \partial y^{\alpha_2} }$.

\item The fractional ordered subspace $H^{\frac{
1}{2}}(\gamma)
$ for $\gamma \subseteq \partial \O$ is defined as
 \begin{align*}
  H^{\frac{1}{2}}(\gamma) := \bigg\{{\psi}\in L^{2}(\gamma) ~\big{|}~~\dfrac{{|{\psi}(x)- {\psi}(y)|}}{{| x - y |}} \in L^2(\gamma \times \gamma) \bigg\}
 \end{align*}
 and is equipped with the norm
 \begin{align*}
 \|\psi\|_{{H^{\frac{1}{2}}}(\gamma) }: =\bigg (\| \psi\|^2_{{L^2
  }(\gamma)} +  \int\limits_{\gamma}\int\limits_{\gamma}  \dfrac{|{\psi}(x) - \psi(y) |^2}{{\lvert x - y \rvert}^{2}}~dx dy \bigg)^{\frac{1}{2}}.
 \end{align*}
 \item ${\mathcal{C}^k_c}(D)$ refers to the space of $k$ times continuously differentiable functions with compact support in $D$ where $D \subseteq \bar{\Omega}$.
 \end{itemize}
We refer to the book \cite{Adams200sobolev} for the detailed information on Sobolev spaces.  Next, we present the weak formulation of the unilateral contact problem.  The functional framework well suited to solve Problem (\eqref{problem1}--\eqref{problem3}) consists in working with the subspace $\bm{\cV}$  of $[H^1(\Omega)]^2$ defined as follows
\begin{align*}
\bm{\cV} : = \{\bm{v} \in [H^1(\Omega)]^2~|~ \bm{v}= \bm{0} ~~\text{on}~ \Gamma_D\}.
\end{align*}
Further,  in order to incorporate the non-penetration condition employed on $\Gamma_C$,  we define the non-empty,  closed,  convex set of admissible displacements as
\begin{align*}
\bm{\cK}:=\{\bm{v}\in \bm{\cV}~|~v_n \leq \chi~~ \mbox{on}~~ \Gamma_C \}.
\end{align*}
Throughout this work,  for any Hilbert space/Banach space $X$,  the notation $\bm{X}$ is used to represent the space of vector-valued functions.  Owing to integration by parts and complimentary contact condition \eqref{problem3},  the weak formulation for the unilateral contact problem can be recasted as an elliptic variational inequality of the first kind as follows:
\begin{Problem} \label{problem2}
	Weak Formulation :-
	\par
	\noindent
	\begin{equation}  \label{eq:CVI}
	\begin{cases}
	\text{To find } \bm{u} \in \bm{\cK} \text{ such that} &\\ 
	a(\bm{u},\bm{v}-\bm{u})- B(\bm{v-u})  \geq 0 \quad \forall ~\bm{v}\in \bm{\cK},
	\end{cases}
	\end{equation}
where $a(\cdot,\cdot)$ and $B(\cdot)$ denote the symmetric bilinear form and the linear functional,  respectively and are defined  by  
	\begin{align}
	a(\bm{w},\bm{v})&:= \int\limits_{\Omega} \bm{\Xi}(\bm{w}): \bm{\epsilon}(\bm{v})~dx
	\quad \forall~\bm{w},\bm{v} \in \bm{\mathcal{V}}, \label{def1}\\ 
	B(\bm{v})&:=\int\limits_{\Omega} \bm{f} \cdot \bm{v}~dx + \int\limits_{\Gamma_N} \bm{\pi} \cdot \bm{v}~ds\quad \forall~ \bm{v} \in \bm{\mathcal{V}}. \label{def2}
	\end{align}
\end{Problem}	
\par
\noindent
Here,  for any two matrices $\b{M}=(m_{ij}) \in \mathbb{R}^{2 \times 2}$ and $\b{N} = (n_{ij})  \in \mathbb{R}^{2 \times 2}$,  the notation $``:"$ represents the inner product defined as follows
\begin{align*}
(\b{M}:\b{N})_{ij}:=\sum_{i,j}m_{ij}n_{ij}.~~ 
\end{align*}
Further $(\cdot,\cdot)_{\O}$ and $\langle \cdot,\cdot\rangle_{\Gamma_N}$ denotes $\b{L}^2$ inner product on $\Omega$ and $\Gamma_N$,  respectively.  The unique solvability of variational inequality \eqref{eq:CVI} is a consequence of well known theorem of Lions and Stampacchia \cite{ciarlet2002finite}. 
\vspace{0.3 cm}
\par 
\noindent
For the ease of presentation,  we choose the outward unit normal vector on $\Gamma_C$ to be $\b{e_1}$ where $\b{e_1}$ and $\b{e_2}$ denotes the standard ordered basis of $\mathbb{R}^2$.  Thus,  the admissible set $\b{\mathcal{K}}$ reduces to 
\begin{align*}
\bm{\cK} =\{\bm{v}\in \bm{\cV}~|~v_1 \leq \chi~~ \mbox{on}~~ \Gamma_C \}.
\end{align*} 
Since the contact condition is enforced on the portion of boundary,  therefore the trace of Sobolev spaces will play a crucial role in the subsequent analysis.  Consequently,  we first recall the preliminary results from the trace theory \cite{s1989topics, sobolev2019}. To this end,  we  introduce the trace map $\gamma_0: \bm{W}^{1,q}(\Omega) \rightarrow \bm{L}^q(\Gamma_C),~1 \leq q < \infty$. In particular,  let $\gamma_c: \bm{\mathcal{V}} \longrightarrow {\b{H}^{\frac{1}{2}}(\Gamma_C) }$ be a continuous onto linear trace operator.  Analogously,  we define trace operator for scalar-valued functions  $\tilde{\zeta}_c: H^1(\Omega)\longrightarrow H^{\frac{1}{2}}(\Gamma_C)$.  The surjectivity of trace map $\gamma_c$ \cite{s1989topics} ensures  a continuous right inverse $\hat{\gamma}_c: \b{H}^{\frac{1}{2}}(\Gamma_C) \longrightarrow \b{\cV} $,  i.e.,  $\forall~\b{v} \in \b{H}^{{\frac{1}{2}}}(\Gamma_C)$ we have,  $ (\gamma_c~o~\hat{\gamma}_c)(\b{v}) = \b{v} $ and 
\begin{align*}
\|\hat{\gamma}_c(\b{v})\|_{\b{H}^1(\O)} \leq c_1 \|\b{v}\|_{\b{H}^{{\frac{1}{2}}}(\Gamma_C)}
\end{align*}
where $c_1$ is a positive constant.
\vspace{0.2 cm}
\par
\noindent
We will make a constant use of the dual space of $\b{H}^{\frac{1}{2}}(\Gamma_C)$ denoted by $\b{H}^{-\frac{1}{2}}(\Gamma_C)$ endowed with the norm $\|\cdot\|_{\b{H}^{-\frac{1}{2}}(\Gamma_C)}$ which is defined as
\begin{align*}
\|\bm{L}\|_{\b{H}^{-\frac{1}{2}}(\Gamma_C)} := \underset{\b{w}  \in \b{H}^{\frac{1}{2}}(\Gamma_C),~\b{w}\neq \b{0}}{\text{sup}} \frac{\bm{L}(\b{w})}{~~\|\b{w}\|_{\b{H}^{\frac{1}{2}}(\Gamma_C)}}\quad\forall~ \bm{L} \in {\b{H}^{-\frac{1}{2}}(\Gamma_C)}.
\end{align*}
Let $\langle \cdot,\cdot \rangle_c$ denotes the duality pairing between the space ${\b{H}^{-\frac{1}{2}}(\Gamma_C)}$ and ${\b{H}^{\frac{1}{2}}(\Gamma_C)}$ i.e. $\langle \b{L},\b{w} \rangle_c=\b{L}(\b{w})~\forall~\b{L} \in {\b{H}^{-\frac{1}{2}}(\Gamma_C)}$ and $\b{w}\in{\b{H}^{\frac{1}{2}}(\Gamma_C)}$. Now, we introduce the continuous contact force density $\b{\sigma} \in {\b{H}^{-\frac{1}{2}}(\Gamma_C)}$ which converts the variational inequality \eqref{eq:CVI} into a equation. It is defined as follows
\begin{align} \label{sigma}
\langle \b{\sigma}, \b{w} \rangle_c := B(\hat{\gamma}_c(\b{w}))- a(\b{u},\hat{\gamma}_c(\b{w}))\quad\forall~\b{w}\in {\b{H}^{\frac{1}{2}}}(\Gamma_C).
\end{align}
The following lemma  from the article  \cite{gudi2016posteriori} establish the properties of continuous contact force density.
\begin{lemma} \label{property1}
	Let $\b{u}$ be the solution of the continuous variational inequality \eqref{eq:CVI}. Then,  the following relations hold
\begin{enumerate}[label=(\roman*)]
\item
	$\langle \b{\sigma}, \gamma_c(\b{v})\rangle_c = B(\b{v}) - a(\b{u}, \b{v})\quad \forall~ \b{v}\in\b{\cV}.    $
\item
	$\langle \b{\sigma}, \gamma_c(\b{v})- \gamma_c(\b{u}) \rangle_c \leq 0 \quad \forall~\b{v}\in \b{\cK}.$
	
	\item $\langle \b{\sigma}, \gamma_c(\b{\phi}) \rangle_c \geq 0 \quad \forall~ \b{0} \leq \b{\phi} \in \b{\cV}. $
	\end{enumerate}
\end{lemma}
\noindent
The relations in Lemma \ref{property1} are a direct consequence of the continuous variational inequality \eqref{eq:CVI} and equation \eqref{sigma}. Next we introduce an intermediate space $\b{\cV_0}$ as 
\begin{align*}
\b{\cV_0}:= \{ \b{w}=(w_1,w_2) \in \b{\cV}, ~\tilde{\zeta}_c(w_1) = 0\}.
\end{align*}
As $\b{v}= \b{u} \pm \b{w}^* \in \b{\cK}~\forall~\b{w}^* \in \b{\cV_0}$, therefore the relation $\eqref{eq:CVI}$ stems down to
\begin{align}\label{prop5}
a(\b{u},\b{w}^*)=B(\b{w}^*)\quad\forall~\b{w}^*\in \b{\cV_0}.
\end{align}
Now,  we collect a key representation for continuous contact force density $\b{\sigma}$ in the next remark.
\begin{remark}\label{rem2}
	For each $\b{v} = (v_1,v_2):= \b{z}_1+\b{z}_2$ where $\b{z}_1=(v_1,0)$ and $\b{z}_2=(0,v_2)$, we can rewrite $\langle \b{\sigma}, \gamma_c(\b{v})\rangle_c = \langle {\sigma}_1, \tilde{\zeta}_c({v_1}) \rangle +\langle {\sigma_2},\tilde{\zeta}_c({v_2}) \rangle$ where
	\begin{align*}
	\langle {\sigma}_1,\tilde{\zeta}_c({v_1}) \rangle &:= B(\b{z}_1) - a (\b{u}, \b{z}_1), \\
	\langle {\sigma}_2, \tilde{\zeta}_c({v_2}) \rangle &:= B(\b{z}_2) - a (\b{u}, \b{z}_2). 
	\end{align*}
Therein,  $\langle \cdot, \cdot \rangle$ denotes the duality pairing between $H^{-\frac{1}{2}}(\Gamma_C)$ and $ H^{\frac{1}{2}}(\Gamma_C).$ 
\vspace{0.4 cm}
\par 
\noindent
A use of relation \eqref{prop5} and the fact that $\b{z_2} \in \b{\cV_0}$,  we obtain 
	\begin{align*}
	\langle {\sigma}_2, \tilde{\zeta}_c({v_2}) \rangle &= 0 \quad\forall~ \b{v}=(v_1,v_2) \in \b{\cV}.
	\end{align*}
\end{remark}
\noindent
Thus the representation of $\b{\sigma}$ in Remark \ref{rem2} stems down to 
\begin{align} \label{prop6}
\langle \b{\sigma}, \gamma_c(\b{v})\rangle_c =  \langle {\sigma}_1, \tilde{\zeta}_c({v_1}) \rangle.
\end{align}
\vspace{0.5 mm}
\par 
\noindent
Let us now emphasize on the support of $\b{\sigma}$ which plays a crucial role in pointwise a posteriori error analysis.  For this purpose,  we will introduce the following non-contact set as follows
\begin{align*}
\mathcal{H}:= \{ x \in \Gamma_C : u_1(x) < \chi(x) \}.
\end{align*}
\begin{lemma} \label{sign3}
	It holds that
	\begin{align} \label{support}
	supp(\sigma_1 ) \subset \Gamma_C \setminus \mathcal{H}.
	\end{align}
\end{lemma}
	\begin{proof}
{Let $\b{\phi}=(\phi_1, \phi_2) \in \b{\mathcal{C}^\infty_c}(\mathcal{H})$ be an arbitrary function.  Extending it by $\b{0}$ outside on whole $\partial \Omega $ yields $\b{\phi} \in \b{\mathcal{C}_c}(\partial \Omega )$.  A use of trace map guarantees the existence of a function $\b{\tilde{\phi}} \in \b{H^1}(\O)$ such that $\gamma(\b{\tilde{\phi}} )=\b{\phi}$, where $\gamma : \b{H}^1(\O) \longrightarrow \b{H}^{\frac{1}{2}}(\partial \Omega)$ is the trace map \cite{kesavan1989topics}. Next,  we verify that for any $0 < \epsilon \leq \epsilon_0 $, where $ \epsilon_0$ is a fixed number,  we have $\b{v} = \b{u} \pm \epsilon \b{\tilde{\phi}} \in  \b{\mathcal{K}}$.  } 
\vspace{0.2 cm}
\par
\noindent
{ It is evident that $\b{u} \pm \epsilon \b{\tilde{\phi}} \in  \b{\mathcal{V}} $. 
Therefore,  it suffices to show $u_1(x^{*}) \pm \epsilon {\phi}_1(x^{*}) \leq \chi(x^{*})$ for any $x^{*} \in \Gamma_C$.  To establish its validity,  we will examine two cases 
\noindent
\begin{enumerate}
\item If ${\phi}_1(x^*) >0$, then for $0 < \epsilon \leq \dfrac{\chi(x^*)-u_1(x^*)}{{\phi}_1(x^*)}$, we have $u_1(x^*) \pm \epsilon{\phi}_1(x^*) \leq \chi(x^*)$.  \\

\item If ${\phi}_1(x^*) < 0$, then for $0 < \epsilon \leq \dfrac{u_1(x^*)-\chi(x^*)}{{\phi}_1(x^*)}$, we have $u_1(x^*) \pm \epsilon{\phi}_1(x^*) \leq \chi(x^*)$.
\end{enumerate}
Thus we conclude $\b{v}= \b{u} \pm \epsilon \b{\tilde{\phi}} \in  \b{\mathcal{K}}$.  Using continuous variational inequality \eqref{eq:CVI}, we obtain
\begin{align*}
B(\b{\tilde{\phi}}) - a(\b{u}, \b{\tilde{\phi}})=0.
\end{align*}
Finally utilizing Lemma \ref{property1},  we infer $$\langle \b{\sigma}, \gamma_c(\b{\tilde{\phi}})\rangle_c = \langle \b{\sigma}, \b{{\phi}}\rangle_c = 0.$$This completes the proof.}

	\end{proof}
\section{Quadratic Discontinuous Galerkin FEM} \label{sec3}
\par
\noindent
To approximate Problem \ref{eq:CVI},  we first recall some preliminary results and notations which will be used in the further analysis.  Since $\Omega$ is a bounded polygonal domain,  thus it can be superimposed by rectilinear finite elements.  For a  given mesh parameter $h>0$,  let $\cT_h$ be the partition of $\O$ into regular triangles \cite{ciarlet2002finite} such that
\begin{align*}
\bar{\O}=\underset{T \in \cT_h}{ \bigcup}T.
\end{align*}
In the forthcoming analysis, let $\delta_{ij}~\text{denotes the Kronecker's Delta defined by,}~\delta_{ij}=1~\text{if}~i=j~\text{and}~0$ {otherwise.} Further,  the notation $X \lesssim Y$ is used to represent $X \leq CY$ where $C$ is a positive generic constant independent of the mesh parameter $h$. 
\vspace{0.5cm}
\par
\noindent
We now employ the quadratic discontinuous finite element space for the discrete approximation of the continuous space $\bm{\cV}$ in the following way
\begin{align*}
\b{\mathcal{V}_h}:=\big\{\b{v}_h \in \b{L}^2(\Omega)~|~\b{v}_h|_T \in [\mathbb{P}_2(T)]^2 \quad\forall~ T \in \cT_h \big \}.
\end{align*} 
Therein,  for any $T \in\cT_h$ and $0 \leq r \in \mathbb{Z},~\mathbb{P}_r(T)$ refers to the space of  polynomials of degree at most $r$.  Associated with 
triangulation $\cT_h$,  let $\mathcal{T}_p$ be the union of all elements sharing the node $p$ and $h_p:= \text{diam } \mathcal{T}_p$.  Further, $\mathcal{T}^C_h$ represents the set of triangles sharing the edge with contact boundary. We denote $\cF_h$ to be the set of all edges of $\cT_h$.  For any edge $e \in \cF_h, $ let $h_e$ denotes the diameter of edge $e$.  We define
\begin{align*}
\cF_h^{int}&:=\{ e \in \cF_h: e \subset \Omega \},\\
\cF_h^{b}&:=\{ e \in \cF_h: e \subset \partial\Omega \},\\
\cF_h^{D}&:=\{ e \in \cF_h: e \subset \Gamma_D \}, \\
\cF_h^{	N}&:=\{ e \in \cF_h: e \subset \Gamma_N \},\\
\cF_h^{C}&:=\{ e \in \cF_h: e \subset \Gamma_C \},\\
\cF_h^{0}&:=\cF_h^{int} \cup \cF_h^{D}.
\end{align*}
The set of all vertices of the triangulation $\cT_h$ is denoted by $\cN_h$.  Based on the decomposition of the boundary $\partial \Omega$ into $\Gamma_D, \Gamma_N$ and $\Gamma_C$, we distribute the vertices of the  triangulation $\cT_h$ as follows: $\cN_h^{int}$ is the set of all interior vertices, $\cN_h^D$ is the set of vertices on $\overline{\Gamma}_D$, $\cN_h^N$ is the collection of vertices on $\Gamma_N$ and $\cN_h^C$ is the set of vertices lying on $\overline{\Gamma}_C$.  Let us denote $\cN^T_h$ as the set of all vertices of the element $T$ and $\cN^e_h$ denotes the set of vertices on edge $e \in \mathcal{F}_h$.   Similarly,  we distribute the set of all midpoints of edges $\cM_h$ of $\cT_h$ as 
\vspace{0.1cm}
\begin{align*}
~~~~\cM_h^{int} &:=\text{set of all midpoints of edges lying in}~\O,\\
\cM_h^D&:= \text{set of all midpoints of edges lying on}~ \Gamma_D, \\
\cM_h^N&:= \text{set of all midpoints  of edges lying on}~ \Gamma_N, \\
\cM_h^C&:=\text{set of all midpoints  of edges lying on} ~\Gamma_C. 
\end{align*}
In addition,  the notation $\cM^e_h$ refers to the midpoint of the edge $e \in \cF_h$ and $\cM^T_h$ denotes the set of all midpoints of edges of the element $T$.  For the sake of ease in presentation,  we assume that each element $T \in \mathcal{T}^C_h$ has exactly one potential contact boundary edge.
\vspace{0.3 cm}
\par
\noindent
In the sequel,  we shall define jumps and averages of non-smooth functions across the interfaces.  For that,  we define broken Sobolev space 
\begin{align*}
{H^1}(\O, \mathcal{T}_h):= \{ {v} \in {L^2}(\O)|~ {v}_T:={v}|_{T} \in {H^1}(T) \quad\forall~T \in \mathcal{T}_h \}.
\end{align*}
Let $e \in \cF^{int}_h$ be an interior edge shared by two neighboring elements $T_1$ and $T_2$.  Further,  let ${\bm{n}_{T_1}}$ be the outward unit normal vector on edge $e$ pointing from $T_1$ to $T_2$ and  ${\bm{n}
_{T_2}}= - {\bm{n}_{T_1}}$.  We define jump $\sjump{\cdot}$ and mean $\smean{\cdot}$ on edge $e$ as follows 
\begin{itemize}
\item \textit{For a scalar-valued function ${w} \in {H^1}(\O, \mathcal{T}_h)$, define }
\begin{align*}
\sjump{w} :={w}|_{T_1}{\bm{n}_{T_1}}+{w}|_{T_2}{\bm{n}_{T_2}},\hspace {1.2 cm}
\smean{w} :=  \frac{{w}|_{T_1}+ {w}|_{T_2}}{2}.
\end{align*}
\item \textit{For a vector-valued function $\b{v} \in [{H^1}(\O, \mathcal{T}_h)]^{2}$, define }
\begin{align*}
\hspace{-0.4 cm}\sjump{\b{v}} :=\b{v}|_{T_1}\otimes{\bm{n}_{T_1}}+\b{v}|_{T_2}\otimes{\bm{n}_{T_2}},~~~\hspace {0.4 cm}
\smean{\b{v}} :=  \frac{\bm{v}|_{T_1}+ \bm{v}|_{T_2}}{2}.
\end{align*}
where the dyadic product of two vectors $\b{a}=(a_1,a_2)$ and  $\b{b}=(b_1,b_2)$ is defined as $(\bm{a} \otimes \bm{b})_{ij}:=a_i b_j,~1 \leq i,j \leq 2$.
\vspace{0.2 cm}
\item \textit{For a tensor-valued function $\b{\phi} \in [{H^1}(\O, \mathcal{T}_h)]^{2\times 2}$, define }
\begin{align*}
\sjump{\b{\phi}} :=\b{\phi}|_{T_1}{\bm{n}_{T_1}}+\b{\phi}|_{T_2}{\bm{n}_{T_2}},\hspace {1.2 cm}
\smean{\b{\phi}} :=  \frac{\bm{\phi}|_{T_1}+ \bm{\phi}|_{T_2}}{2}.
\end{align*}
\end{itemize}
Analogously,  for the sake of notational convenience we define jump $\sjump{\hspace{0.05 cm}\cdot\hspace{0.05 cm}}$ and mean $\smean{\cdot}$ on boundary edges also.  For any $e \in \cF^b_h,$  there exists an element $T \in \cT_h$ such that $e \in \partial T~ \cap~ \partial \Omega$.  We set jump $\sjump{\hspace{0.05 cm}\cdot\hspace{0.05 cm}}$ and  average $\smean{\cdot}$ on edge $e$ as follows 
\begin{itemize}
\item \textit{For a scalar-valued function $w \in {H^1}(\O, \mathcal{T}_h)$,  we set}
$\sjump{w} := w{\bm{n_e}}~\text{and}~\smean{w} := w$.
\item \textit{For a vector-valued function $\b{v} \in [{H^1}(\O, \mathcal{T}_h)]^{2}$,  we set}
$\sjump{\b{v}} :=\b{v}\otimes{\bm{n_e}}~\text{and}~
\smean{\b{v}} :=  \bm{v}.$
\item \textit{For a tensor-valued function $\b{\phi} \in [{H^1}(\O, \mathcal{T}_h)]^{2\times 2}$,  we set }
$\sjump{\b{\phi}} :=\b{\phi}{\bm{n_e}}~\text{and}~
\smean{\b{\phi}} :=  \bm{\phi}$.
\end{itemize}
\vspace{0.2 cm}
\par 
\noindent
Here,  $\b{n_e}$ is the outward unit normal attributed to edge $e$. In addition,  for a  given function $\bm{v}$, we define $\bm{v}^+:= \text{max}\{\bm{v}, \bm{0}\}$ to be non-negative part of $\bm{v}$.  Furthermore,  for any $\bm{v} \in \bm{\cV_h}$ and $T \in \cT_h$,  define $\bm{\epsilon}_h(\bm{v})$ as $\bm{\epsilon}_h(\bm{v})|_T= \bm{\epsilon}(\bm{v})$ on $T$ and $\bm{\Xi}_h(\bm{v})= 2 \mu \bm{\epsilon}_h(\bm{v}) + \kappa(tr (\bm{\epsilon}_h(\bm{v}))  \textbf{I}$.  
\begin{remark}
The definitions of jumps and averages defined in the preceding paragraph are also applicable to functions in $\b
{W}^{1,p}(\O),~1 \leq p < 2$. 
\end{remark}
\vspace{0.2 cm}
\par
\noindent
Next we introduce the discrete analogue $\b{\mathcal{K}_h}$  of the set $\b{\cK}$ in the following way
\begin{align*}
\b{\mathcal{K}_h} : = \bigg\{ \b{v}_h:=(v_{h,1}, v_{h,2}) \in \b{\mathcal{V}_h}~|~ \int\limits_e v_{h,1}~ds \leq \int\limits_e \chi~ds \quad\forall~e \in \cF_h^C \bigg \}.
\end{align*}
\par
\noindent
It can be realised that in the discrete set $\b{\cK_h}$, the non-penetration condition is incorporated in the form of integral constraints.  
The discrete formulation associated to Problem  \ref{problem2} reads as follows:
\begin{Problem}
	Discrete variational inequality :-
	\par
	\noindent
	\begin{equation}  \label{eq:DVI}
	\begin{cases}
	\text{To find } \bm{u}_h \in \bm{\cK_h} \text{ such that} &\\ 
	M_{DG}(\bm{u}_h,\bm{v}_h-\bm{u}_h) -B(\bm{v}_h-\bm{u}_h) \geq 0 \quad \forall~ \bm{v}_h \in \bm{\cK_h},
	\end{cases}
	\end{equation}
\end{Problem}
\noindent
where the mesh dependent bilinear form $M_{DG}(\cdot, \cdot)$ defined on $\b{\cV_h} \times \b{\cV_h} $ is given by
	 $$M_{DG}(\bm{u}_h,\bm{v}_h):=a_h(\bm{u}_h,\bm{v}_h)+L_{DG}(\bm{u}_h,\bm{v}_h)$$ 
with $a_h(\bm{u}_h,\bm{v}_h):= \sum\limits_{ T \in \cT_h}\int\limits_{T} \bm{\Xi}(\bm{u}_h): \bm{\epsilon}(\bm{v}_h)~dx$ and
	 $L_{DG}(\bm{u}_h,\bm{v}_h)$ consists of consistency and stability terms \cite{wang2011discontinuous,gudi2016posteriori,walloth2019reliable}.
\vspace{0.5 cm}
\par
\noindent
\textit{\textbf{Examples of DG Methods}}
\vspace{0.2 cm}
\par
\noindent
In this article,  we have performed the analysis that is applicable to several DG formulations that have been enumerated in \cite{wang2011discontinuous}.  However,  for the sake of simplicity we have chosen to focus the forthcoming analysis on three specific methods: SIPG,  NIPG and IIPG.  While the results are relevant to other DG formulations,  we believe that analyzing these three methods in detail will provide meaningful insights that can be generalized to other DG methods as well.  To this end,  the different choices for $L_{DG}(\cdot,\cdot)$ are listed as follows:
\bigbreak
	\begin{itemize} 
		\item \textit{ \textbf{SIPG method}} \cite{ciarlet2002finite,wang2011discontinuous}:
		\begin{align*}
		\hspace{2mm}
		L_{DG}(\b{u}_h,\b{v}_h) &=  - \int \limits_{e \in \mathcal{F}^0_h } \sjump{\b{u}_h}:\smean{\b{\bm{\Xi}}_h(\b{v}_h)} ~ds - \int \limits_{e \in \mathcal{F}^0_h } \sjump{\b{v}_h}:\smean{{\bm{\Xi}}_h(\b{u}_h)}~ ds \\
		& \hspace{0.2cm} + \int \limits_{e \in \cF_h^0}\eta h_e^{-1}\sjump{\b{u}_h}:\sjump{\b{v}_h}~ds, \quad 	
		\end{align*}
$\text{for } \b{u}_h, \b{v}_h \in 	\b{\mathcal{V}_h} \text{ and } \eta > \eta_0>0.$
		\bigbreak
		\item \textit{ \textbf{NIPG method}} \cite{ciarlet2002finite,wang2011discontinuous}:
		\begin{align*}
		\hspace{2mm}
		L_{DG}(\b{u}_h,\b{v}_h) &= \int \limits_{e \in \mathcal{F}^0_h } \sjump{\b{u}_h}:\smean{\b{\bm{\Xi}}_h(\b{v}_h)}  ~ds - \int \limits_{e \in \mathcal{F}^0_h } \sjump{\b{v}_h}:\smean{\b{\bm{\Xi}}_h(\b{u}_h)} ~ds \\
		& + \int \limits_{e \in \cF_h^0}\eta h_e^{-1}\sjump{\b{u}_h}:\sjump{\b{v}_h}~ds, \quad 
		\end{align*}
for $ \b{u}_h$, $\b{v}_h \in 	\b{\mathcal{V}_h} \text{ and } \eta >0.$
		\bigbreak
		\item \textit{ \textbf{IIPG method}} \cite{ciarlet2002finite,wang2011discontinuous}:
		\begin{align*}
		\hspace{2mm}
		L_{DG}(\b{u}_h,\b{v}_h) =  - \int \limits_{e \in \mathcal{F}^0_h } \sjump{\b{v}_h}:\smean{\b{\bm{\Xi}}_h(\b{u}_h)} ~ds + \int \limits_{e \in \cF_h^0}\eta h_e^{-1}\sjump{\b{u}_h}:\sjump{\b{v}_h}~ds,
		\end{align*}
		for $\b{u}_h, \b{v}_h \in 	\b{\mathcal{V}_h}$ and $\eta > 0.$
\end{itemize}

\noindent
We end this section by recalling the following technical results that will be required in the subsequent analysis \cite{brenner2007mathematical}.
\begin{lemma}(Inverse inequalities) \label{Lemama}
	Let $1 \leq s,t \leq \infty$ and $\bm{v}_h \in \bm{\cV_h}$. Then, it holds that
	\begin{enumerate}
		\item $	\|\bm{v}_h\|_{\bm{W}^{m,s}(T)} \lesssim h^{l-m}_T h_T^{2(\frac{1}{s}-\frac{1}{t})} \|\bm{v}_h\|_{\bm{W}^{l,t}(T)} \quad \forall~T~\in \cT_h$, $l \leq m$,
		\item $	\|\bm{v}_h\|_{\bm{L}^{\infty}(T)} \lesssim h^{-1}_T \|\bm{v}_h\|_{\bm{L}^2(T)} \quad \forall~ T \in \cT_h$,
		\item $	\|\bm{v}_h\|_{\bm{L}^{\infty}(e)} \lesssim h^{-\frac{1}{2}}_e \|\bm{v}_h\|_{\bm{L}^2(e)} \quad \forall~ e \in \cF_h$.
	\end{enumerate}	
\end{lemma}
\noindent
\begin{lemma}(Trace inequality) \label{Lemmmma}
	Let $p \in [1, \infty)$ and $ e\in \partial T$ where $T \in \cT_h$. Then,  for any $\bm{\phi} \in \bm{W}^{1,p}(T),$ the following holds 
	\begin{align*}
	\|\bm{\phi}\|^p_{\bm{L}^p(e)} \lesssim h_e^{-1} \Big( \|\bm{\phi}\|^p_{\bm{L}^p(T)} + h_e^p \|\bm{\nabla} \bm{\phi}\|^p_{\bm{L}^p(T)} \Big).
	\end{align*}
\end{lemma}
\section{ Discrete Lagrange Multiplier}\label{sec4}
\noindent
This section introduces the discrete counterpart of continuous contact force density $\b{\sigma}$.  For that,  we introduce the space $$\b{\mathcal{W}_h}:= \{\b{w}_h \in \b{L}^2(\Gamma_C)~|~\b{w}_h|_e \in [\mathbb{P}_0(e)]^2 \quad\forall~ e \in \cF_h^C\}$$ and the map $\Pi_h: \b{L}^2(\Gamma_C) \rightarrow \b{\mathcal{W}_h} $, which is defined by
\begin{align*}
 \Pi_h(\b{w})|_e:= \frac{1}{h_e} \int\limits_e \b{w}~ds \quad \forall~\b{w} \in \b{L}^2(\Gamma_C)~\text{and}~ e \in \cF_h^C.
\end{align*}
Using the interpolation map $ \Pi_h$,  we define the map $\b{\phi}_h: \b{\mathcal{V}_h} \rightarrow \b{\mathcal{W}_h}$ by $\b{\phi}_h(\b{v}_h):= \Pi_h(\Gamma_h(\b{v}_h))$, where $\Gamma_h: \b{\mathcal{V}_h} \rightarrow \b{\mathcal{Z}_h}:= \{\b{z}_h \in \b{L}^2(\Gamma_C)~|~\b{z}_h|_e \in [\mathbb{P}_2 (e)]^2 \quad\forall~e \in \cF_h^C\}$ denotes the trace map.
\vspace{.5 cm}
\par
\noindent
Now,  let us turn to the lemmas essential for the further analysis.
\begin{lemma} \label{prop1}
	The map $\b{\phi}_h :\b{\mathcal{V}_h} \rightarrow \b{\mathcal{W}_h}$ defined by
	\begin{align*}
	\b{\phi}_h(\b{v}_h):= \Pi_h(\Gamma_h(\b{v}_h))
	\end{align*}
	is surjective.
\end{lemma} 
\begin{proof}
	Let $\mathcal{O}:=\big\{d^j,~j \in \{1,2, \cdots, n\}\big\}$ denotes the enumeration of edges on $\Gamma_C$.  Thus $$\Gamma_C= \bigcup\limits_{j=1}^{n}d^j.$$ Recalling the definition of space $\b{\mathcal{W}_h}$,  for any $\b{w}_h \in \b{\mathcal{W}_h}$, we say that for $1\leq j \leq n$,  $$\b{w}_h|_{d^j} :=(\alpha^{j,1},\alpha^{j,2})$$ where $\alpha^{j,1}$ and  $\alpha^{j,2}$ are constant quantities.  Now,  we construct a test function $\b{v}^*_h \in \b{\mathcal{V}_h}$ as follows
\begin{align*}
\begin{split}
\begin{aligned}
\b{v}^{*}_h = \begin{cases} &\hspace{-0.2 cm}(\alpha^{j,1},\alpha^{j,2})~~~~~~~~~\text{if}~T\in T^j, \\
&(0,0)~~~~~~~\hspace{1 cm}\text{if}~T \in \cT_h \setminus \bigcup\limits_{j=1}^{n}T^j. \end{cases}
\end{aligned}
\end{split}
\end{align*}
where $ T^j$ refers to the triangle such that $d^j \subseteq \partial T^j$.  Finally,  we have $\b{\phi}_h(\b{v}^*_h)= \Pi_h(\Gamma_h(\b{v}^*_h))=\b{w}_h$. Hence, the claim follows.
\end{proof}
\par
\noindent
For the sake of clarity,  we will utilize the following depiction for the map $\b{\phi}_h = (\phi_{h,1}, \phi_{h,2})$ where
\begin{align*}
\phi_{h,1}(\b{v}_h)|_e= \frac{1}{h_e}\int \limits_{e} v_{h,1}~ds~~\text{and}~~
\phi_{h,2}(\b{v}_h)|_e= \frac{1}{h_e}\int \limits_e v_{h,2}~ds 
\end{align*}
for all $e\in\mathcal{F}^C_h$ and $\b{v}_h \in  \b{\mathcal{V}_h}$.
\vspace{0.3 cm}
\par
\noindent 
As a consequence of Lemma \ref{prop1}, we observe that there exists a continuous right inverse $$\b{\phi}_h^{-1}: \b{\mathcal{W}_h} \rightarrow \b{\mathcal{V}_h}$$ which is defined as $\b{\phi}_h^{-1}(\b{w}_h)=\b{v}_h^*$, where $\b{v}_h^*$ is stated in the proof of Lemma \ref{prop1}. 
\vspace{0.5 cm}
\par
\noindent
Next, we introduce the discrete contact force density $\b{\sigma}_h:=(\sigma_{h,1},\sigma_{h,2}) \in \b{\mathcal{W}_h}$ as
\begin{align} \label{def:dcfd}
( \b{\sigma}_h, \b{w}_h ):= B(\b{\phi}_h^{-1}(\b{w}_h))- M_{DG}(\b{w}_h, \b{\phi}_h^{-1}(\b{w}_h)) \quad \forall~ \b{w}_h \in \b{\mathcal{W}_h}
\end{align}
where $(\cdot, \cdot)$ denotes $\b{L}^2$ inner product on $\Gamma_C$.
\vspace{0.3 cm}
\par 
\noindent
\begin{remark} \label{prop2}
For the space $\b{\mathcal{V}^0_h}:= \{\b{v}_h \in \b{\mathcal{V}_h}~:~\phi_{h,1}(\b{v}_h)=0\}$,  it is useful to note that
	\begin{align} \label{sign}
	B(\b{z}_h) - M_{DG}(\b{u}_h,\b{z}_h)=0 ~\text{for any}~ \b{z}_h \in \b{\mathcal{V}_h^0}
	\end{align}
as $ \b{u}_h \pm \b{z}_h \in \b{\mathcal{K}_h}$.  The space $\b{\mathcal{V}_h^0}$ will play an essential role in proving well-definedness of map $\b{\sigma}_h$.
\end{remark}
\noindent
The upcoming lemma is a consequence of equation \eqref{def:dcfd} and Remark \ref{prop2}.
\begin{lemma} \label{ksign}
	The map $\b{\sigma}_h \in \b{\mathcal{W}_h}$ defined by relation \eqref{def:dcfd} is well-defined and satisfies
	\begin{align}
	( \b{\sigma}_h, \b{v}_h ) = B(\b{v}_h)- M_{DG}(\b{u}_h, \b{v}_h) \quad \forall~ \b{v}_h \in \b{\mathcal{V}_h}.
	\end{align}
\end{lemma}
\begin{proof}
Firstly, we prove that the map $\b{\sigma}_h \in \b{\mathcal{W}_h}$ is well-defined.  Let $\b{w}_h, \b{r}_h \in \b{\mathcal{W}_h}$ be such that $\b{w}_h=\b{r}_h$.  There exist $\b{v}_1,  \b{v}_2 \in \b{\mathcal{V}_h}$ such that $\b{\phi}_h(\b{v}_1)=\b{w}_h$ with $\b{\phi}_h^{-1}(\b{w}_h)=\b{v}_1$ and  $\b{\phi}_h(\b{v}_2)=\b{r}_h$ with $\b{\phi}_h^{-1}(\b{r}_h)=\b{v}_2$.  Let $\b{z}:=\b{v}_1-\b{v}_2$ and using $\b{w}_h=\b{r}_h$, we have $\b{\phi}_h(\b{z})=0$. By Remark \ref{prop2}, we conclude $( \b{\sigma}_h, \b{w}_h)=( \b{\sigma}_h, \b{r}_h)$. For the second part, let $\b{v}_h \in \b{\mathcal{V}_h}$, then we have
\\
	\begin{align*}
( \b{\sigma}_h, \b{v}_h )&= ( \b{\sigma}_h, \b{\phi}_h(\b{v}_h) ) \\ & = B\big(\b{\phi}_h^{-1}(\b{\phi}_h(\b{v}_h))\big )-M_{DG}\big(\b{u}_h, \b{\phi}_h^{-1}(\b{\phi}_h(\b{v}_h))\big) \\ & = B\big(\b{\phi}_h^{-1}(\b{\phi}_h(\b{v}_h))- \b{v}_h\big )- M_{DG}\big(\b{u}_h, \b{\phi}_h^{-1}(\b{\phi}_h(\b{v}_h))-\b{v}_h\big) + B(\b{v}_h) - M_{DG}(\b{u}_h,\b{v}_h) \\ & =B(\b{v}_h) - M_{DG}(\b{u}_h,\b{v}_h)
	\end{align*}
by using the fact that $\b{\phi}_h^{-1}(\b{\phi}_h(\b{v}_h))- \b{v}_h \in \b{\mathcal{V}_h^0}$ and equation \eqref{sign}.
\end{proof}
\vspace{0.2 cm}
\noindent
For any $\b{v}_h=(v_{h,1},v_{h,2}) \in \b{\mathcal{V}_h}$, we make use of the following representation
	\begin{align*}
( \b{\sigma}_h, \b{\phi}_h(\b{v}_h))= ( {\sigma}_{h,1}, \phi_{h,1}(\b{v}_h))_{L^2(\Gamma_C)} + ({\sigma}_{h,2}, \phi_{h,2}(\b{v}_h)_{L^2(\Gamma_C)}, 
	\end{align*}
with the components
\begin{align*}	
({\sigma}_{h,1}, \phi_{h,1}(\b{v}_h))_{L^2(\Gamma_C)}  &:= B((v_{h,1},0))- M_{DG}(\b{u}_h,(v_{h,1},0)),\\ ({\sigma}_{h,2}, \phi_{h,2}(\b{v}_h))_{L^2(\Gamma_C)} &:=B((0,v_{h,2}))- M_{DG}(\b{u}_h,(0,v_{h,2})).
\end{align*}
\noindent
Observe that $(0,v_{h,2}) \in \b{\mathcal{V}^0_h}$, thus using Remark \ref{sign} we find $({\sigma}_{h,2}, \phi_{h,2}(\b{v}_h))_{L^2(\Gamma_C)} = 0.$
\begin{remark} \label{sign4}
As $\b{\sigma}_h \in \b{\mathcal{W}_h} \subseteq \b{L}^2(\Gamma_C)$,  it can also be identified as an element of $\b{H}^{-\frac{1}{2}}(\Gamma_C)$ by defining
	\begin{align} \label{prop3}
	\langle \b{\sigma}_h, \b{v}\rangle_c = \int\limits_{\Gamma_C} \b{\sigma}_h \cdot\b{v}~ds \quad\forall~\b{v}\in\b{H}^{\frac{1}{2}}(\Gamma_C). 
	\end{align}
\end{remark}

\vspace{0.3 cm}
\par
\noindent
Next, we obtain the sign properties of discrete contact force density $\b{\sigma}_h$ which play a key role in proving the reliability estimates. 
\begin{lemma}\label{sign2}
	Let $e \in \cF_h^C$.  Then,  the discrete contact force density $\b{\sigma}_h=(\sigma_{h,1},\sigma_{h,2})$ satisfies
	\begin{align*}
	\begin{split}
	\begin{aligned}
	\sigma_{h,1}|_e \geq 0~~\text{ and }~~~~
	\sigma_{h,2}|_e = 0
	\end{aligned} \quad \forall~ e \in \cF_h^C.
	\end{split}
	\end{align*}
In addition,  if $e^* \in \mathcal{C}_h^N : = \bigg\{ e \in \cF_h^C ~:~ \int\limits_e u_{h,1}~ds < \int\limits_e \chi~ds ~ \bigg \}~\text{then}~~\sigma_{h,1}|_{e^*}=0$.
\end{lemma}
\begin{proof}
In order to establish the proof of the lemma we construct an appropriate test function $\b{v}_h\in \b{{\mathcal{V}}_h}$ by following these steps: 
Let $e \in \mathcal{F}^C_h$ be an arbitrary edge on the potential contact boundary.  Then choose a node $p$ from $ \mathcal{N}^e_h \cup \mathcal{M}^e_h$ and define 
\begin{align*}
\begin{split}
\begin{aligned}
\b{v}_h(z) = \begin{cases} &(1,0)~~~\text{if}~ z=p, \\
&(0,0)~~~\text{if}~z\in (\mathcal{N}_h \cup \mathcal{M}_h)\setminus \{p\}. \end{cases}
\end{aligned}
\end{split}
\end{align*}
Observe that $\b{u}_h-\b{v}_h\in \b{{\mathcal{K}}_h}$.  This implies $(\b{\sigma}_h,  \b{v}_h) \geq 0$
as $B(\b{v}_h)\geq M_{DG}(\b{u}_h, \b{v}_h)$ using discrete variational inequality \eqref{eq:DVI}.
Also,  
\begin{align}\label{1.10}
0 \leq(\b{\sigma}_h,  \b{v}_h)=(\b{\sigma}_h, \b{\phi}_h(\b{v
}_h))=  ( {\sigma}_{h,1}, \phi_{h,1}(\b{v}_h))_{L^2(\Gamma_C)} = \int\limits_e \sigma_{h,1}  \phi_{h,1}(\b{v}_h)~ds =   \sigma_{h,1}|_e \int \limits_e v_{h,1}~ds. 
\end{align}
Thus,  we have $\sigma_{h,1}|_e\geq 0$ since $ \int \limits_e v_{h,1}~ds > 0$.  Similarly upon defining a suitable test function $\b{v}_h \in \b{\mathcal{V}_h}$ one can  show that $\sigma_{h,2}|_e = 0 \quad\forall~e \in \cF_h^C.$ Now for $e^* \in \mathcal{C}_h^N, $ choose an arbitrary node $p \in \mathcal{N}^h_e \cup \mathcal{M}^h_e$.  For sufficiently small $\delta>0$,  define  
\begin{align*}
\b{v}_h = \b{u}_h \pm \delta\phi_p\b{e_1}, 
\end{align*}
where $\phi_p$ refers to  $\mathbb{P}_2$ Lagrange nodal basis function corresponding to the node $p$.  Note that $\b{u}_h \pm \b{v}_h \in \b{\mathcal{K}_h}$.  Thus  $(\b{\sigma}_h, \b{v}_h) = B(\b{v}_h) - M_{DG}(\b{u}_h, \b{v}_h)) = 0$.  Further using relation \eqref{1.10} we find
\begin{align*}
(\b{\sigma}_h, \b{v}_h)  = \sigma_{h,1}|_e \int \limits_e v_{h,1}~ds.
\end{align*}
This in turn shows that $\sigma_{h,1}|_e = 0$
as $ \int \limits_e v_{h,1}~ds > 0$.  Consequently $\sigma_{h,1}|_e = 0 \quad\forall~e\in \mathcal{C}^N_h.$
\end{proof}
\noindent
Next we introduce the following map $\widehat{{\b{\sigma}}}_h : \bm{L}^1(\Gamma_C) \longrightarrow \mathbb{R}$ by
	\begin{align} \label{sign6}
	\widehat{\b{\sigma}
	}_h(\bm{v}):= \int\limits_{\Gamma_C} \b{\sigma}_h \cdot\b{v}~ds \quad\forall~\b{v} \in  \bm{L}^{1}(\Gamma_C).  
	\end{align}
Observe that the map defined in relation \eqref{sign6} is well-defined as $\b{\sigma}_h |_e \in [\mathbb{P}_0(e)]^2 \quad\forall~ e \in \mathcal{F}^C_h.$
\begin{remark}\label{rem121}
In particular, observe that if $\bm{w} \in \bm{\cV}$ then 
	\begin{align}
	\widehat{\b{\sigma}
	}_h(\gamma_0(\bm{w}))= \langle\b{\sigma}_h,  \gamma_c(\b{w})\rangle_c.
	\end{align}
\end{remark}
\par
\noindent
\subsection{The Smoothing Operator}
In this segment we construct a continuous approximation to the discontinuous solution $\b{u}_h$.  This operator plays a crucial role while carrying out the analysis using DG methods.  The manner in which it is built can vary depending on the specific requirements \cite{Brenner2003KornsIF}.  Here,  we construct the smoothing map $\bm{E}_h=({E}_{h,1}, {E}_{h,2})$ on the discrete space 
\begin{align*}
\bm{E}_h:\b{\mathcal{V}_h} \rightarrow\b{\mathcal{V}^C_h}:= \{\bm{v}_h \in  \b{H}^1(\Omega)~|~\bm{v}_h|_T \in [\mathbb{P}_2(T)]^2 ~\forall~ T \in \cT_h~ \text{and}~ \bm{v}_h = \bm{0} ~~\text{on}~ \Gamma_D\}
\end{align*}
using the technique of averaging as follows:
\begin{itemize}
\item  For $p \in \cN_h^D \cup \cM^D_h$,  we set $\bm{E}_h\b{u}_h(p) = \b{0}$.

\item For the remaining nodes define
$\b{E}_h\b{u}_h(p) = \dfrac{1}{|\mathcal{T}_p|} \sum\limits_{T \in \cT_p} \b{u}_h|_T(p)$.
\end{itemize}
\par
\noindent
Next we state the approximation properties of enriching map $\bm{E}_h$ which are crucial for the subsequent analysis.  
\begin{lemma} \label{enriching_map}
	Let $\bm{E}_h:\b{\mathcal{V}_h} \rightarrow \b{\mathcal{V}^C_h}$ be the smoothing operator. Then,  for any $\bm{v}_h \in \bm{\cV_h$}, the following holds
	\begin{align} 
	\max_{T\in \cT_h}~h_T\|\bm{\nabla} (\bm{E}_h\b{v}_h-\bm{v}_h)\|_{\bm{L}^{\infty}(T)}  &\lesssim  ~\|\sjump{\bm{u}_h}\|_{\bm{L}^{\infty}(\cF^0_{h})},   \label{GFGFF} \\  \max_{T\in \cT_h}~\| \bm{E}_h\b{v}_h-\bm{v}_h\|_{\bm{L}^{\infty}(T)} & \lesssim  ~\|\sjump{\bm{u}_h}\|_{\bm{L}^{\infty}(\cF^0_{h})}. \label{GFGF}
\end{align}
\end{lemma}
\noindent
The proof of the lemma  can be accomplished using the similar ideas as in [Lemma 3.5 ,\cite{KP:2022:Signorini}] and therefore it is omitted.
\vspace{0.3 cm}
\par
\noindent
In order to facilitate the further analysis,  we require a quasi-interpolation operator $$\Phi_h : \b{L}^1(\Omega) \rightarrow \b{\mathcal{V}_h}$$ which is defined in the following way.
\begin{align}
\Phi_h\b{v}|_T(x)=\Phi_T\b{v}(x):= \sum_{i=1}^2 \sum_{p \in \mathcal{N}_h^T \cup \mathcal{M}_h^T} \phi_p^T(x) \int\limits_{T} \psi_p^T(z)v_i(z) \b{e}_i~dz \quad\forall~\b{v}=(v_1,v_2) \in   \b{L}^1(\Omega)
\end{align}
where $\{ \phi_p^T(x)~; p \in \mathcal{N}_h^T \cup \mathcal{M}_h^T, 	~ T \in \mathcal{T}_h \}$ denotes the Lagrange basis for the space $\b{\mathcal{V}_h}$ and $\psi_p^T(x)$ is $L^2(T)$ dual basis of $\phi_p^T(x)$~\cite{demlow2012pointwise}.
\vspace{0.3 cm}
\par
\noindent
In the next lemma, we state the stability and approximation results of interpolation operator $\Phi_h$ which can be established using Bramble Hilbert lemma (for details refer to \cite{demlow2012pointwise}).
\begin{lemma} \label{approx}
	Let $T \in \cT_h$ and $T^*= \{T' \in \cT_h: T' \cap T \neq \emptyset \}$. Then,  for $\{t \in \mathbb{N} \cup \{0\};~t \leq 3 \},$ the following holds  
	\begin{align}
	|\Phi_h \b{\psi}|_{\bm{W}^{t,1}(T)} &\lesssim |\b{\psi}|_{\bm{W}^{t,1}(T^*)}, \label{DR1}\\
	\|{\b{\psi} -\Phi_h\b{\psi}}\|_{\bm{W}^{s,1}(T)} & \lesssim h_T^{t-s}|\b{\psi}|_{\bm{W}^{t,1}(T^*)}~~ s\in \{0,1,2\}, \label{DR2}	\end{align}
	where $\b{\psi} \in \b{W}^{t,1}(T)$.
\end{lemma} 
\section{A Posteriori Error Estimates} \label{sec5}
\noindent
In this section, the methodology for pointwise a posteriori error analysis is laid out. Therein,  the residual error estimator is introduced  and the article's  first main result  i.e.,  the reliability of the error estimator is stated.  Firstly,  we extend the bilinear form $M_{DG}$ so that it allows for testing non-discrete functions less regular than $\bm{H}^1(\Omega)$. Afterwards, we discuss the existence of Green's matrix for the divergence type operators and collect some regularity results which will be crucial in the subsequent analysis. 
\vspace{0.2 cm}
\par 
\noindent
For $p > 2$,  $1 \leq q < 2$ s.t. $\frac{1}{p}+ \frac{1}{q}=1$, we define
\begin{align}
\bm{\mathcal{X}}:= \bm{W}^{1,q}(\bar{ \Omega}) + \b{\mathcal{V}_h} \quad \text{and}\quad \bm{\mathcal{Y}}:=  \bm{W}^{1,p}(\bar{ \Omega}) + \b{\mathcal{V}_h}.
\end{align}
We introduce the extended bilinear form $M^*_{DG}:  \bm{\mathcal{Y}} \times  \bm{\mathcal{X}} \rightarrow \mathbb{R} $ as
\begin{align*}
M^{*}_{DG}(\bm{v},\bm{w})&=  \int\limits_{\Omega} \bm{\Xi}(\bm{v}): \bm{\epsilon}(\bm{w}) ~dx + L^*_{DG}(\bm{v},\bm{w}) \quad \forall~ \b{v} \in  \bm{\mathcal{Y}}, ~\b{w} \in  \bm{\mathcal{X}},
\end{align*}
where
\begin{align} \label{eq3}
L^*_{DG}(\bm{v},\bm{w})&= - \sum_{e \in \cF^{0}_h} \int\limits_e \smean {{\bm{\Xi}}(\Phi_h(\bm{v}))}: \sjump{\bm{w}}~ds + \theta \sum_{e \in \cF^{0}_h} \int\limits_e \sjump{\bm{v}}: \smean{{\bm{\Xi}}(\Phi_h(\bm{w})}~ds  \\  & \hspace{1cm}+ \sum_{e \in \cF^{0}_h } \frac{\eta}{h_e} \int\limits_e \sjump{\bm{v}}: \sjump{\bm{w}}~ds \quad \forall~ \b{v} \in  \bm{\mathcal{Y}}, \b{w} \in  \bm{\mathcal{X}},
\end{align}
with $\eta>0$ as the penalty parameter \cite{wang2011discontinuous} and $\theta \in \{-1,0,1\}$. 
\begin{remark} \label{proj}
By definition of the dual basis \cite{demlow2012pointwise}, we observe that $ \int_{T} \psi_p^T(x)\phi_r^T(x)dx=\delta_{pr}$, therefore, it holds that $\Phi_h\b{w}=\b{w} \quad\forall~ \b{w} \in \b{\cV_h}$, i.e., $\Phi_h$ is a projection map. Hence, we conclude that $M^{*}_{DG}(\bm{v},\bm{w})= M_{DG}(\bm{v},\bm{w}) \quad \forall~\bm{v}, \bm{w} \in \b{\mathcal{V}_h}.$
\end{remark}
\subsection{Existence of Green's matrix}
\noindent
 In the analysis,  we will make use of the Green's matrix for the divergence type operators.  Let $\bm{\Lambda}^{z_0}:=(\Lambda_{ij}(\cdot, z_0))_{i,j=1}^2$ be the Green's matrix associated with \eqref{problem1} having singularity at $z_0$. Let $H_{ij\alpha\beta}=\mu (\delta_{ij} \delta_{\alpha\beta}+ \delta_{j\alpha} \delta_{i\beta})+ \kappa \delta_{i\alpha} \delta_{j\beta} $  and $\delta_{z_0}$ be the Dirac $\delta-$distribution with a unit mass at $z_0$, then $\bm{\Lambda}^{z_0}$ satisfies the following equations
\begin{align} \label{eq:GR}
-\sum_{j=1}^{2} \sum_{\alpha, \beta =1}^{2} D_{\alpha} (H_{ij\alpha\beta}D_{\beta}\Lambda_{jk}^{z_0}) &= \delta_{ik}\delta_{z_0}(\cdot) \quad \text{in}~~ \Omega, \\ \bm{\Lambda}^{z_0} &= \b{0} \hspace{1.3cm} \text{on}~\Gamma_D, \notag \\  {\zeta_{ij\alpha\beta} D_{\beta} \Lambda_{jk}(\cdot,z_0) n_{\alpha}} & ={0} \hspace{1.3cm}  \text{on}~ \Gamma_N \cup \Gamma_C. \notag
\end{align}
Further,  for all $x,y \in \Omega$ s.t. $x\neq y$ we have the following regularity estimates \cite{dolzmann1995estimates,hofmann2007green}.
\begin{enumerate}
	\item For $1 \leq i, j \leq 2,$ we have ${\Lambda}^{z_0}_{i,j} \in {W}^{1,s}(\Omega) \quad \forall~ 1 \leq s < 2,$
	\item $|\bm{\Lambda}^{z_0}(x,y)| \lesssim~ \text{ln} \frac{C}{|x-y|}$. 
\end{enumerate}
\begin{lemma} \label{lemama}
	For $T \in \cT_h$, the following holds
	\begin{align*}
	\sum_{j=1}^{2} \sum_{T \in \cT_h}	h^{-1}_T|\Lambda_{k,j}^{z_0}|_{W^{1,1}(T)} \lesssim 1+  |ln(h_{min})|^2.
	\end{align*}
\end{lemma}
\begin{proof}
	We refer to the article \cite{KP:2022:QuadLinfSignorini} for the proof.
\end{proof}
\par
\noindent	
For any set $D \subseteq \bar{ \Omega}$, we introduce the notation  $\bm{\mathcal{N}}_D$ as
\begin{align} \label{4.26}
\bm{\mathcal{N}}_D:= \{\bm{w} \in \bm{W}^{2,1}(D) \cap \bm{\cV}~ \text{such that}~ \bm{\epsilon(w)} \bm{n}=0~\text{on}~ \partial D\},
\end{align} 
equipped with the norm $\|\b{v}\|_{\bm{\mathcal{N}}_D}:=\|\bm{v}\|_{\b{W}^{1,1}(D)} + |\bm{v}|_{\b{W}^{2,1}(D)}.$ We assume $\bm{\mathcal{N}}:={\bm{\mathcal{N}}_D}$ for $\mathcal{D}= \bar{ \Omega}$.
\noindent
Next,  we introduce the residual error estimators
\begin{align*}
\eta_1&= \max_{T \in \cT_h}~h_T^2 \|\b{f} + \b{div}\b{\Xi}(\b{u}_h)\|_{\bm{L}^{\infty}(T)},\\
\eta_2&= \max_{e \in \cF_h^{int}}~h_e \|\sjump{\bm{\Xi}(\bm{u}_h)}\|_{\bm{L}^{\infty}(e)} ,\\
\eta_3&= \max_{e \in \cF_h^N}~h_e \|\b{\pi}-\b{\Xi}(\b{u}_h)\b{n}_e\|_{\bm{L}^{\infty}(e)},\\
\eta_4&= \max_{e \in \cF_h^C }~h_e { \|\b{\Xi}(\b{u}_h)\b{n}_e+ \b{\sigma}_h\|_{\bm{L}^{\infty}(e)}},\\
\eta_5&= \max_{e \in \cF^{0}_h}~ \|
\sjump{\b{u}_h}\|_{\bm{L}^{\infty}(e)}.
\end{align*}
\par
\noindent
The total residual error estimator $\mathcal{E}_h$ is defined by
	\begin{equation}\label{total1}
	\mathcal{E}_h:= (1+ |\text{ln} (h_{min})|^2 )\Big(\sum_{i=1}^{5} \eta_i \Big)  + {
		\|(E_{h,1}(\b{u}_{h})-\chi)^{+}\|_{L^{\infty}(\Gamma_C)} +  \|(\chi-E_{h,1}(\b{u}_{h}))^+\|_{L^{\infty}(\{\sigma_{h,1}>0\})} }.
	\end{equation}
In the following subsection we establish the reliability of error estimator $\mathcal{E}_h$.  Therein,  we aim to find the upper bound of $\mathcal{E}_h$.
\subsection{Reliability of the error estimator}
\begin{theorem}\label{thm:rel11} (Reliability)
	Let $(\bm{u}, {\bm{\sigma}})$ satisfy equations \eqref{eq:CVI} and \eqref{sigma},  respectively and let $(\bm{u}_h, \bm{\sigma}_h)$ be solution of equations \eqref{eq:DVI} and \eqref{def:dcfd}, respectively. Then, the following reliability estimate holds
	\begin{equation}  \label{rel1}
	\max{\{\|\bm{u}-\bm{u}_h\|_{\bm{L}^{\infty}(\Omega)}~,~ \|{\bm{\sigma}}-\bm{\sigma}_h\|_{-2,\infty, \Gamma_C} \}} \lesssim \mathcal{E}_h,
	\end{equation}
where the norm $\|\cdot\|_{-2,\infty, \Gamma_C}$ is defined as 
\begin{align} \label{4.27}
{\|\cdot\|_{-2,\infty, \Gamma_C}} := \sup \{\langle \bm{\cdot},\gamma_c(\bm{w}) \rangle_c : \bm{w} \in \bm{\mathcal{N}}~,~\|\bm{w}\|_{\bm{\mathcal{N}}} \leq 1\}.
\end{align}
\end{theorem}
\par 
\noindent
The steps for proving Theorem \ref{thm:rel11} are broken down into the following subsections.  As previously indicated, the analysis is inspired by the concepts presented in the articles \cite{demlow2012pointwise,khandelwal2022pointwiseq} with the  modifications made to account for the approximation's discontinuity. To achieve this, we begin with the definition of the residual functional,  i.e. $\bm{R}_h$ (equation \eqref{Gal1}) and the conforming portion of the approximate solution $\bm{u}_h$, i.e.  $\bm{E}_h(\bm{u}_h)$ and in the view of correcting it, we construct the upper and lower barrier functionals for the continuous solution.  In the remaining article,  we will denote  enriching map $\bm{E}_h(\bm{u}_h)$ by $\bm{u}^{enr}_h$. 
\vspace{0.3 cm}
\par
\noindent
In the subsection \ref{sub4.2}, we note that the main idea to introduce the barrier functionals for the continuous solution $\bm{u}$ is to improve the discrete solution $\bm{u}^{enr}_h$ with the help of $\b{u}^{corr}$ and the terms appearing due to the approximation of boundary contributions.  In the subsection \ref{4.3}, the proof of the main reliability estimate (Theorem \ref{thm:rel11}) is achieved by providing a maximum norm estimate for $\b{u}^{corr}$ in terms of the local estimators $\eta_i,~ i = 1, 2, 3, 4,5$.
\subsection{Barrier Functions for the continuous solution} \label{sub4.2}
\noindent We introduce the residual functional $\bm{R}_h \in \bm{\mathcal{X}}^*$ which serves the same purpose as the residual functional for unconstrained problems and is defined by
\begin{align} \label{Gal1}
\langle \bm{R}_h , \bm{v} \rangle_{-1,1}:= 	\tilde{a}(\bm{u}- \bm{u}^{enr}_h, \bm{v}) +  \tilde{\bm{\sigma}}{(\bm{v})}- \widehat{\b{\sigma}}_h(\gamma_0(\bm{v}))  \quad\forall ~\bm{v} \in \bm{\mathcal{X}},
\end{align}
where $\tilde{a}(\bm{w}, \bm{v}):=\sum\limits_{T \in \mathcal{T}_h} \int\limits_{T} \bm{\Xi}(\bm{w}): \bm{\epsilon}(\bm{v})~dx~\forall~\bm{w},\bm{v} \in \bm{\mathcal{X}}$ and $\langle \cdot, \cdot \rangle_{-1,1}$ denotes the duality pairing between the spaces $\bm{\mathcal{X}}^*$ and $\bm{\mathcal{X}}$. { Here, the extended continuous contact force density $\tilde{\bm{\sigma}}\in \bm{\mathcal{X}}^*$ is defined by
\begin{align}\label{eq:sigmadef1}
\tilde{\bm{\sigma}}(\bm{v}):= \tilde{B}(\bm{v}) -\tilde{a}(\bm{u},\bm{v})\quad\forall~\bm{v}\in \bm{\mathcal{X}}.
\end{align}}
where $\tilde{B}(\bm{v}):= (\bm{f},\bm{v})_{\O} + \langle \bm{\pi},\bm{v} \rangle_{\Gamma_N}~ \forall~ \bm{v} \in \bm{\mathcal{X}}.$
\noindent 
\begin{remark}\label{rem1}
In particular, if $\bm{w}, \bm{v} \in \bm{\cV}$, then 
	\begin{align}
	{ \tilde{\bm{\sigma}}(\bm{v})= \langle \b{\sigma}, \gamma_c(\bm{v}) \rangle_{c},~~
\tilde{a}(\bm{w}, \bm{v})={a}(\bm{w}, \bm{v})~~\text{and}~~\tilde{B}(\bm{v})= {B}(\bm{v}) .}
	\end{align}
\end{remark}
\noindent
Let $\b{u}^{corr}:=(u_1^{corr},u_2^{corr}) \in \b{\cV}$ be the corrector function which satisfies the following variational formulation 
\begin{align} \label{eq:CRF}
\int\limits_{\Omega} \bm{\Xi}(\b{u}^{corr}): \bm{\epsilon}(\bm{v}) ~dx= \langle \b{R}_h, \b{v} \rangle_{-1,1}  \quad \forall~\b{v} \in \b{\cV}.
\end{align} 
The well-posedness of relation \eqref{eq:CRF} follows from the Lax-Milgram lemma.  With the use of corrector function $\b{u}^{corr}$,  we introduce the upper and lower barrier functional corresponding to the continuous solution $\b{u}$ in the following way
\begin{align}
\b{u_{\circ}}&=\b{u}^{enr}_h+\b{u}^{corr}-\b{c}-\b{s}, \label{eq:lowerb}\\
\b{u^{\circ}}&=\b{u}^{enr}_h+\b{u}^{corr}+\b{c}+\b{y}, \label{eq:upperb}
\end{align}
where,
\begin{align*}
&\bullet  \b{c}~\text{denotes the vector-valued function defined such that its each component is given by}~\|\b{u}^{corr}\|_{\b{L}^{\infty}(\bar{ \Omega}) }, \\ &\hspace{0.4cm} \text{where}~ \|\b{u}^{corr}\|_{\b{L}^{\infty}(\bar{ \Omega})}:=\max \Big\{\|u_1^{corr}(x)\|_{L^{\infty}(\bar{ \Omega})},\|u_2^{corr}(x)\|_{L^{\infty}(\bar{ \Omega})}\Big\}, \\ &\bullet \b{s}~ \text{is the vector-valued function having each component as}~ \|(u_{h,1}^{enr}-\chi)^{+}\|_{L^{\infty}(\Gamma_C)},\\ &\bullet \b{y}~ \text{refers to the vector-valued function with each component as}~ \|(\chi-u_{h,1}^{enr})^+\|_{L^{\infty}(\{\sigma_{h,1}>0\})} .
\end{align*}

\noindent
In the next lemma we prove that $\b{u^{\circ}}$ and $\b{u_{\circ}}$ are upper and lower barriers of $\b{u}$.
\begin{lemma} \label{lem:uplo}
	The following holds
	\begin{align*}
	\b{u_{\circ}} \leq \b{u} \quad \text{and}\quad \b{u^{\circ}} \geq \b{u},
	\end{align*}
where $\b{u_{\circ}}$ and $\b{u^{\circ}}$are defined in equations \eqref{eq:lowerb} and \eqref{eq:upperb},  respectively.
\end{lemma}
\begin{proof}
	We start by proving $\b{u}_{\circ} \leq \b{u}$. Set $\b{z}=\max\{\b{u_{\circ}}-\b{u},\b{0}\}$. Observe that $\b{z}|_{\Gamma_D}=\b{0} \iff (\b{u_{\circ}}-\b{u})|_{\Gamma_D} \leq \b{0}$. Also
	\begin{align*}
	(\b{u_{\circ}}-\b{u})|_{\Gamma_D}&=(\b{u}^{enr}_h+(\b{u}^{corr}-\b{c})-\b{s}-\b{u})|_{\Gamma_D} \leq (\b{u}^{enr}_h-\b{u})|_{\Gamma_D} = \b{0}.
	\end{align*}
	Thus $\b{z} \in \b{\cV}$.  Further taking into account Poincar$\acute{e}$ inequality \cite{s1989topics} it is sufficient to show that $\|\b{\nabla}  \b{z}\|_{\b{L}^2(\Omega)}=0$. By the use of $\b{\cV}$ ellipticity of bilinear form $a(\cdot,\cdot)$, equations \eqref{eq:CRF}, \eqref{Gal1} together with Remarks \ref{rem121},  \ref{rem1} and relation \ref{prop6} we derive
	\begin{align}
	\|\b{\nabla} \b{z}\|^2_{\b{L}^2(\Omega)}& \lesssim a(\b{z},\b{z}) \notag\\ &= a(\b{u_{\circ}}-\b{u},\b{z}) = a(\b{u}^{enr}_h+\b{u}^{corr}-\b{u},\b{z}) \notag \\ & =a(\b{u}^{enr}_h-\b{u},\b{z})+ \langle \b{R}_h, \b{z} \rangle_{-1,1} \notag \\&=\tilde{\bm{\sigma}}( \bm{z}) - \widehat{\bm{\sigma}}_h(\gamma_0(\bm{z}))  \notag \\ 
&= \langle \b{\sigma}, \gamma_c(\bm{z}) \rangle_{c} - \langle {\bm{\sigma}}_h , \gamma_c(\bm{z}) \rangle_c  \notag \\ 
&= \langle {\sigma}_1, \tilde{\zeta}_c({z_1}) \rangle -\int\limits_{\Gamma_C} \b{\sigma}_h \cdot \b{z}~ds, \label{eq1}
	\end{align}
where $\b{z}=(z_1,z_2)$.  From the definition, we have $\b{z} \geq \b{0}  \implies z_1 \geq 0.$ Using  Lemma \ref{sign2}, we deduce
	\begin{align*}
	\int\limits_{\Gamma_C} \b{\sigma}_h \cdot \b{z}~ds = \int\limits_{\Gamma_C} \sigma_{h,1} z_1~ds  \geq 0.
	\end{align*}
	Therefore, \eqref{eq1} reduces to
	\begin{align*} 
	\|\b{\nabla} \b{z}\|^2_{\b{L}^2(\Omega)} & \leq  \langle {\sigma}_1, \tilde{\zeta}_c({z_1}) \rangle.
	\end{align*}
	Now we try to show that $\text{supp}(\sigma_1)~  \cap~ \text{supp}(\tilde{\zeta}_c({z_1})) = \emptyset $. Suppose, we have $x^* \in \Gamma_C$ and $z_1(x^*)>0$ then using the definition of $\b{z}$, we have
	\begin{align*}
	u_1|_{\Gamma_C}(x^*) < u_{\circ,1}|_{\Gamma_C}(x^*) &= u_{h,1}^{enr}|_{\Gamma_C}(x^*) + u^{corr}_1|_{\Gamma_C}(x^*) -\|\b{u}^{corr}\|_{\b{L}^{\infty}(\bar{ \Omega})} -\|(u_{h,1}^{enr}-\chi)^{+}\|_{L^{\infty}(\Gamma_C)}, \\  & \leq u_{h,1}^{enr}|_{\Gamma_C}(x^*)-\|(u_{h,1}^{enr}-\chi)^{+}\|_{L^{\infty}(\Gamma_C)}, \\ & \leq \chi(x^*).
	\end{align*}
	Therefore, we have $\text{supp}(\tilde{\zeta}_c({z_1})) \subset \{u_1 < \chi\}$. Using Lemma \ref{sign3}, we conclude that $\text{supp}(\sigma_1)~\cap~(\text{supp}(\tilde{\zeta}_c({z_1})) = \emptyset $.
	\vspace{0.2 cm}
	\par
	\noindent
	Next, we prove $\b{u^{\circ}} \geq \b{u}$.  The proof proceeds in a similar manner.  Let $\b{z}:=  \max\{\b{u}-\b{u^{\circ}},\b{0}\}$. Again, we will show that $\b{z} \in \b{\mathcal{V}}$. First, observe that $\b{z}|_{\Gamma_D}=\b{0}$, as we have
	\begin{align*}
	(\b{u}-\b{u^{\circ}})|_{\Gamma_D}&=(\b{u}-\b{u}^{enr}_h-\b{u}^{corr}-\b{c}-\b{y})|_{\Gamma_D} \leq (\b{u}-\b{u}^{enr}_h)|_{\Gamma_D} = \b{0}.
	\end{align*}
	Using the coercivity of $a(\cdot,\cdot)$, definition of $\b{u^{\circ}}$, equations \eqref{eq:CRF}, \eqref{Gal1}, Remarks \ref{sign4},  \ref{rem121},  Lemma \ref{property1} and Lemma \ref{sign2},  it holds that
	\begin{align*} 
	\|\b{\nabla} \b{z}\|^2_{\b{L}^2(\Omega)}& \lesssim a(\b{z},\b{z}) \notag \\ &= a(\b{u}-\b{u^{\circ}},\b{z})= a(\b{u}-\b{u}^{enr}_h-\b{u}^{corr},\b{z})\notag \\ & = a(\b{u}-\b{u}^{enr}_h,\b{z}) - \langle \bm{R}_h , \bm{z} \rangle_{-1,1} \\ &= 
 \widehat {\bm{\sigma}}_h(\gamma_0(\bm{z}))- \tilde{\bm{\sigma}}(\bm{z})
 \\ &
=\langle {\bm{\sigma}}_h , \gamma_c(\bm{z}) \rangle_c- \langle{\bm{\sigma}}, \gamma_c (\bm{z})\rangle_c  
 \\ &  \leq \langle {\bm{\sigma}}_h , \gamma_c(\bm{z}) \rangle_c \\ & = \int\limits_{\Gamma_C} \sigma_{h,1}  q_1~ds = \sum\limits_{e \in \cF_h^C} \int\limits_{e} \sigma_{h,1}  q_1~ds,
	\end{align*}
	where $\gamma_c(\b{z}):=\b{q}=(q_1,q_2)$. Let $e \in \cF_h^C$ be any arbitrary edge. It is enough to prove that $\int\limits_{e} \sigma_{h,1} q_1~ds=0$. To prove this, we first show for any $e \in \cF_h^C$, if there exists $x^* \in e$ such that $q_1(x^*) > 0$ then $\sigma_{h,1} =0$ on $e$. We prove it by contradiction. Assume, there exists $x^* \in e$ such that $q_1(x^*)>0$ and $\sigma_{h,1}(x^*)>0$ on $e$. Then, we have
	\begin{align*}
	q_1(x^*)>0 \implies u_1|_{\Gamma_C}(x^*) &> u^{\circ}_1|_{\Gamma_C}(x^*) \\  & \geq u_{h,1}^{enr}|_{\Gamma_C}(x^*) + u^{corr}_1|_{\Gamma_C}(x^*) +\|\b{u}^{corr}\|_{\b{L}^{\infty}(\bar{\Omega})} \\ & \hspace{0.5cm}+ \|(\chi-u_{h,1}^{enr})^{+}\|_{L^{\infty}(\{\sigma_{h,1}>0\})}  \\ & \geq u_{h,1}^{enr}|_{\Gamma_C}(x^*)  + \|(\chi-u_{h,1}^{enr})^{+}\|_{L^{\infty}(\{\sigma_{h,1}>0\})}  \geq \chi(x^*)
	\end{align*}
	which is a contradiction. Finally,  if $q_1=0$ on $e \in \cF_h^C $, then $ \int\limits_{e} \sigma_{h,1} q_1~ds =0$ and if there exists $x^* \in e$  such that $q_1(x^*)=z_1|_{\Gamma_C}(x^*)>0$ then $\sigma_{h,1}=0$ on $e$. Therefore, $\sum\limits_{e \in \cF_h^C} \int\limits_{e} \sigma_{h,1}q_1~ds =0$. Hence, we conclude the proof of the lemma.
\end{proof}
\noindent
With the help of Lemma \ref{lem:uplo}, we deduce the following result.
\begin{lemma} \label{lem4.5}
	It holds that
	\begin{align*}
	\|\b{u}-\b{u}_h\|_{\b{L}^{\infty}(\Omega)} \lesssim 2\|\b{u}^{corr}\|_{\b{L}^{\infty}(\bar{ \Omega}) } + \|(u_{h,1}^{enr}-\chi)^{+}\|_{L^{\infty}(\Gamma_C)} + \|(\chi-u_{h,1}^{enr})^+\|_{L^{\infty}(\{\sigma_{h,1}>0\})} + \eta_5 \notag.
	\end{align*}
\end{lemma}
\noindent
Besides the bound on the corrector function, i.e.  $\|\b{u}^{corr}\|_{\b{L}^{\infty}(\bar{ \Omega}) }$,  all terms involved in the right hand side of the last estimate depend only on given data and the discrete solution
and are thus computable.  It remains to estimate $\|\b{u}^{corr}\|_{\b{L}^{\infty}(\bar{ \Omega}) }$ in terms of computable quantities. 
\noindent
\subsection{Bound on $\|\b{u}^{corr}\|_{\b{L}^{\infty}(\bar{ \Omega}) }$} \label{4.3}	It is essential to bound the term $\|\b{u}^{corr}\|_{\b{L}^{\infty}(\bar{ \Omega}) }$ to get the desired upper bound on the error term $\|\b{u}-\b{u}_h\|_{\b{L}^{\infty}(\Omega)}$. Now, let $k \in \{1,2\}$ and $z_0 \in \bar{ \Omega}$ be such that $\|\b{u}^{corr}\|_{\b{L}^{\infty}(\bar{ \Omega}) }=|u_k^{corr}(z_0)|$. Using equations (\ref{eq:GR}) and (\ref{eq:CRF}), we deduce
\begin{align} \label{ww}
\|\b{u}^{corr}\|_{\b{L}^{\infty}(\bar{ \Omega}) }=|u_k^{corr}(z_0)|= \langle \bm{R}_h , \b{\Lambda}_k^{z_0} \rangle_{-1,1},
\end{align}
where $\bm{\Lambda}_k^{x_0}$ is the $k^{th}$-column of the Green matrix $\bm{\Lambda}^{z_0}$. Using the definition of the residual functional \eqref{Gal1}, Remark \ref{property1}, \eqref{def1}, \eqref{def2}, rearrangement of terms and integration by parts, we have
\begin{align} \label{property3}
\langle \bm{R}_h , \b{\Lambda}_k^{z_0} \rangle_{-1,1} &= \tilde{a}(\b{u}-\b{u}^{enr}_h,\b{\Lambda}_k^{z_0})+ \tilde{\bm{\sigma}}(\b{\Lambda}_k^{z_0}) - \widehat{\bm{\sigma}}_h(\gamma_0(\b{\Lambda}_k^{z_0})) \nonumber\\ &= \tilde{a}(\b{u}, \b{\Lambda}_k^{z_0}) + \tilde{\bm{\sigma}}(\b{\Lambda}_k^{z_0}) - \tilde{a}(\b{u}^{enr}_h,\b{\Lambda}_k^{z_0})- \widehat{\bm{\sigma}}_h(\gamma_0(\b{\Lambda}_k^{z_0})) \nonumber \\ 
&= \tilde{B}(\b{\Lambda}_k^{z_0}) - \tilde{a}(\b{u}^{enr}_h,\b{\Lambda}_k^{z_0})- \widehat{\b{\sigma}}_h(\gamma_0(\b{\Lambda}_k^{z_0})) \nonumber \\
 & = \bigg( \tilde{B}( \b{\Lambda}_k^{z_0}- \Phi_h(\b{\Lambda}_k^{z_0})) -\tilde{a}(\b{u}_h,\b{\Lambda}_k^{z_0}- \Phi_h(\b{\Lambda}_k^{z_0}))- \widehat{ \b{\sigma}}_h (\gamma_0(\b{\Lambda}_k^{z_0}-\Phi_h(\b{\Lambda}_k^{z_0})))  \bigg) \nonumber \\ & \hspace{0.5cm} + \tilde{B}(\Phi_h(\b{\Lambda}_k^{z_0})) -\tilde{a}(\b{u}_h,\Phi_h(\b{\Lambda}_k^{z_0})) -\widehat{ \b{\sigma}}_h(\gamma_0(\Phi_h(\b{\Lambda}_k^{z_0}))) -\tilde{a}(\b{u}^{enr}_h-\b{u}_h,\b{\Lambda}_k^{z_0})  \nonumber \\ & = \tilde{B}( \b{\Lambda}_k^{z_0}- \Phi_h(\b{\Lambda}_k^{z_0})) + \sum\limits_{T\in \mathcal{T}_h} \int\limits_{T} \b{div}\b{\Xi}(\b{u}_h)\cdot (\b{\Lambda}_k^{x_0}- \Phi_h(\b{\Lambda}_k^{z_0}))~dx \nonumber \\ & \hspace{0.4cm}  - \sum\limits_{ T \in \cT_h} ~\int\limits_{\partial T} (\bm{\Xi}(\bm{u}_h) \bm{n}) \cdot (\b{\Lambda}_k^{z_0}- \Phi_h(\b{\Lambda}_k^{z_0})) ~ds - \widehat{ \b{\sigma}}_h( \gamma_0(\b{\Lambda}_k^{z_0}-\Phi_h(\b{\Lambda}_k^{z_0}))) \nonumber \\ & \hspace{0.5cm} + \tilde{B}(\Phi_h(\b{\Lambda}_k^{z_0})) -M^*_{DG}(\b{u}_h,\Phi_h(\b{\Lambda}_k^{z_0})) + L^*_{DG}(\b{u}_h,\Phi_h(\b{\Lambda}_k^{z_0}))-\widehat{ \b{\sigma}}_h(\gamma_0(\Phi_h(\b{\Lambda}_k^{z_0}))) \nonumber \\ & \hspace{0.5cm} -\tilde{a}(\b{u}^{enr}_h-\b{u}_h,\b{\Lambda}_k^{z_0}). \end{align}
For a tensor-valued function $\bm{w}$ and vector-valued $\bm{v}$, we have the below identity \cite{wang2011discontinuous}
\begin{align} \label{property2}
\sum\limits_{ T \in \cT_h}~ \int\limits_{\partial T} (\bm{w} \bm{n}) \cdot \bm{v} ~ds & = \sum\limits_{e \in \cF^{int}_h} \int\limits_e \sjump{\bm{w}} \cdot \smean{\bm{v}}~ds + \sum\limits_{e \in \cF_h} \int\limits_e \smean{\bm{w}} : \sjump{\bm{v}}~ ds.
\end{align} 
Therefore, using \eqref{property2} and using the fact that $M^*_{DG}(\b{v}_h,\b{w}_h)=M_{DG}(\b{v}_h,\b{w}_h)~\text{and}~\tilde{B}(\b{v}_h)={B}(\b{v}_h)~\forall~\b{v}_h,\b{w}_h \in \b{\cV_h}$, we can rewrite \eqref{property3} as follows
\begin{align} \label{pro2}
\langle \bm{R}_h , \b{\Lambda}_k^{z_0} \rangle_{-1,1}&= \tilde{B}( \b{\Lambda}_k^{z_0}- \Phi_h(\b{\Lambda}_k^{z_0})) + \sum\limits_{T\in \mathcal{T}_h} \int\limits_{T} \b{div}\b{\Xi}(\b{u}_h)\cdot (\b{\Lambda}_k^{x_0}- \Phi_h(\b{\Lambda}_k^{z_0}))~dx \notag \\ & \hspace{0.3cm}  - \sum\limits_{e \in \cF^{int}_h} \int\limits_e \sjump{\bm{\Xi}(\bm{u}_h)} \cdot \smean{\b{\Lambda}_k^{z_0}- \Phi_h(\b{\Lambda}_k^{z_0})}~ds - \sum\limits_{e \in \cF_h} \int\limits_e \smean{\bm{\Xi}(\bm{u}_h)} : \sjump{\b{\Lambda}_k^{z_0}- \Phi_h(\b{\Lambda}_k^{z_0})}~ ds \notag \\ & \hspace{0.3cm} + \underbrace{\bigg[B(\Phi_h(\b{\Lambda}_k^{z_0})) -M_{DG}(\b{u}_h,\Phi_h(\b{\Lambda}_k^{z_0})) -\widehat{ \b{\sigma}}_h( \gamma_0(\Phi_h(\b{\Lambda}_k^{z_0}))) \bigg]}_\text{Term1} \notag \\ & \hspace{0.3cm} - \widehat{\b{\sigma}}_h( \gamma_0(\b{\Lambda}_k^{z_0}-\Phi_h(\b{\Lambda}_k^{z_0}))) + L^*_{DG}(\b{u}_h,\Phi_h(\b{\Lambda}_k^{z_0}))  -\tilde{a}(\b{u}^{enr}_h-\b{u}_h,\b{\Lambda}_k^{z_0}). 
\end{align}
Using Lemma \ref{ksign}, we have Term 1 to be zero. By adding and subtracting the term $L^*_{DG}(\b{u}_h,\b{\Lambda}_k^{z_0})$ and inserting the definition of $B(\cdot, \cdot)$ in \eqref{pro2}, we get 
\begin{align} \label{eq4}
\langle \bm{R}_h , \b{\Lambda}_k^{z_0} \rangle_{-1,1} & = \sum\limits_{T\in \mathcal{T}_h} \int\limits_{T} (\b{f} + \b{div}\b{\Xi}(\b{u}_h))\cdot (\b{\Lambda}_k^{x_0}- \Phi_h(\b{\Lambda}_k^{z_0}))~dx
\notag \\ &\hspace{0.7cm}+\sum\limits_{e \in \cF_h^N}\int\limits_{e}(\b{g}-\b{\Xi}(\b{u}_h)\b{n}_e)\cdot (\b{\Lambda}_k^{z_0}- \Phi_h(\b{\Lambda}_k^{z_0}))~ds \nonumber \\ &\hspace{0.7cm}-\sum\limits_{e\in \cF_h^C}\int_{e}({\b{\Xi}}(\b{u}_h)\b{n}_e+ \b{\sigma}_h) \cdot (\b{\Lambda}_k^{z_0}- \Phi_h(\b{\Lambda}_k^{z_0}))~ds  \notag \\ & \hspace{0.7cm}- \sum\limits_{e \in \cF^{int}_h} \int\limits_e \sjump{\bm{\Xi}(\bm{u}_h)} \cdot \smean{\b{\Lambda}_k^{z_0}- \Phi_h(\b{\Lambda}_k^{z_0})}~ds \nonumber \\ & \hspace{0.7cm} - \underbrace{\bigg(L^*_{DG}\big(\b{u}_h,\b{\Lambda}_k^{z_0}-\Phi_h(\b{\Lambda}_k^{z_0})\big) + \sum\limits_{e \in \cF^{0}_h} \int\limits_e \smean{\bm{\Xi}(\bm{u}_h)} : \sjump{\b{\Lambda}_k^{z_0}- \Phi_h(\b{\Lambda}_k^{z_0})}~ ds\bigg)}_\text{Term  2} \nonumber\\ & \hspace{0.7cm} \underbrace{-\tilde{a}(\b{u}^{enr}_h-\b{u}_h,\b{\Lambda}_k^{z_0}) + L^*_{DG}(\b{u}_h,\b{\Lambda}_k^{z_0})}_\text{Term 3}. 
\end{align}
Next, we deal with the Term 2 and Term 3 separately. Firstly, using the definition of $ L^*_{DG}(\cdot, \cdot)$ in Term 2 together with Remark \ref{proj} and the fact $\Phi_h(\b{\Lambda}_k^{z_0}-\Phi_h(\b{\Lambda}_k^{z_0}))=0$, we have 
\begin{align}
L^*_{DG}(\b{u}_h,\b{\Lambda}_k^{z_0}-\Phi_h(\b{\Lambda}_k^{z_0}))&= - \sum_{e \in \cF^{0}_h} \int\limits_e \smean{{\bm{\Xi}}(\Phi_h(\b{u}_h))} : \sjump{\b{\Lambda}_k^{z_0}-\Phi_h(\b{\Lambda}_k^{z_0})}~ds \nonumber \\ & \hspace{0.6 cm} + \theta \sum\limits_{e \in \cF^{0}_h} \int\limits_e \sjump{\b{u}_h}: \smean{{\bm{\Xi}}(\Phi_h(\b{\Lambda}_k^{z_0}-\Phi_h(\b{\Lambda}_k^{z_0}))}~ds  \nonumber \\  & \hspace{0.6cm}+ \sum\limits_{e \in \cF^{0}_h } \frac{\eta}{h_e} \int\limits_e \sjump{\b{u}_h} : \sjump{\b{\Lambda}_k^{z_0}-\Phi_h(\b{\Lambda}_k^{z_0})}~ds \\ &= - \sum_{e \in \cF^{0}_h } \int\limits_e \smean{\bm{\Xi}(\b{u}_h)}
: \sjump{\b{\Lambda}_k^{z_0}-\Phi_h(\b{\Lambda}_k^{z_0})}~ds  \nonumber \\  & \hspace{0.6cm}+ \sum\limits_{e \in \cF^{0}_h} \frac{\eta}{h_e} \int\limits_e \sjump{\b{u}_h}: \sjump{\b{\Lambda}_k^{z_0}-\Phi_h(\b{\Lambda}_k^{z_0})}~ds.  \label{eq2}
\end{align}	
Therefore, we have the following representation of Term 2. 
\begin{align}
\text{Term 2} = \sum\limits_{e \in \cF^{0}_h} \frac{\eta}{h_e} \int\limits_e \sjump{\b{u}_h}: \sjump{\b{\Lambda}_k^{z_0}-\Phi_h(\b{\Lambda}_k^{z_0})}~ds. \label{eq6}
\end{align} 
Next, we focus on Term 3. Using the definition of $\tilde{a}(\cdot,\cdot)$ and $L^*_{DG}(\cdot,\cdot)$ and the fact that $\sjump{\b{\Lambda}_k^{z_0}}=0$, it holds that
\begin{align}
\text{Term 3}&=\sum\limits_{T\in \mathcal{T}_h} \int\limits_{T}  \bm{\Xi}(\b{u}^{enr}_h-\b{u}_h): \bm{\epsilon}(\b{\Lambda}_k^{z_0}) ~dx  - \sum_{e \in \cF^{0}_h} \int\limits_e \smean{{\bm{\Xi}}(\Phi_h(\b{u}_h))}:\sjump{ \b{\Lambda}_k^{z_0}}~ds  \nonumber\\ & \hspace{0.7cm}+ \theta \sum_{e \in \cF^{0}_h} \int\limits_e \sjump{\b{u}_h}: \smean{{\bm{\Xi}}(\Phi_h(\b{\Lambda}_k^{z_0})}~ds \nonumber  + \sum\limits_{e \in \cF^{0}_h} \frac{\eta}{h_e} \int\limits_e \sjump{\b{u}_h}: \sjump{\b{\Lambda}_k^{z_0}}~ds \nonumber \\ & =  \sum\limits_{T\in \mathcal{T}_h} \int\limits_{T}  \bm{\Xi}(\b{u}^{enr}_h-\b{u}_h): \bm{\epsilon}(\b{\Lambda}_k^{z_0}) ~dx + \theta \sum_{e \in \cF^{0}_h} \int\limits_e \sjump{\b{u}_h}: \smean{{\bm{\Xi}}(\Phi_h(\b{\Lambda}_k^{z_0})}~ds. \label{eq5}
\end{align}
With the help of equations \eqref{pro2}, \eqref{eq4},   \eqref{eq6}, \eqref{eq5},  \eqref{GFGFF} together with Lemmas \ref{Lemmmma}, \ref{approx}, \ref{enriching_map}, \ref{lemama} and H\"older's inequality, the following holds
\begin{align*}
\langle \bm{R}_h , \b{\Lambda}_k^{z_0} \rangle_{-1,1} & = \sum\limits_{T\in \mathcal{T}_h} \int\limits_{T} (\b{f} + \b{div}\b{\Xi}(\b{u}_h))\cdot (\b{\Lambda}_k^{x_0}- \Phi_h(\b{\Lambda}_k^{z_0}))~dx
\notag \\ &\hspace{0.7cm}+\sum\limits_{e \in \cF_h^N}\int\limits_{e}(\b{g}-\b{\Xi}(\b{u}_h)\b{n}_e)\cdot (\b{\Lambda}_k^{z_0}- \Phi_h(\b{\Lambda}_k^{z_0}))~ds \nonumber \\ &\hspace{0.7cm}-\sum\limits_{e\in \cF_h^C}\int\limits_{e}({\b{\Xi}}(\b{u}_h)\b{n}_e+ \b{\sigma}_h) \cdot (\b{\Lambda}_k^{z_0}- \Phi_h(\b{\Lambda}_k^{z_0}))~ds  \notag \\ & \hspace{0.7cm}- \sum\limits_{e \in \cF^{int}_h} \int\limits_e \sjump{\bm{\Xi}(\bm{u}_h)} \cdot \smean{\b{\Lambda}_k^{z_0}- \Phi_h(\b{\Lambda}_k^{z_0})}~ds + \theta \sum\limits_{e \in \mathcal{F}^{0}_h} \int\limits_e \sjump{\b{u}_h}: \smean{{\bm{\Xi}}(\Phi_h(\b{\Lambda}_k^{z_0})}~ds \\ & \hspace{0.7cm} + \sum\limits_{e \in \cF^{0}_h} \frac{\eta}{h_e} \int\limits_e \sjump{\b{u}_h}:\sjump{\b{\Lambda}_k^{z_0}-\Phi_h(\b{\Lambda}_k^{z_0})}~ds  + \sum\limits_{T\in \mathcal{T}_h} \int\limits_{T}  \bm{\Xi}(\b{u}^{enr}_h-\b{u}_h): \bm{\epsilon}(\b{\Lambda}_k^{z_0}) ~dx  \\ & \lesssim \eta_1  \sum_{T\in \mathcal{T}_h} h_T^{-2}\|\b{\Lambda}_k^{z_0}- \Phi_h(\b{\Lambda}_k^{z_0})\|_{\b{L}^1(T)} + \eta_2 \sum_{e \in \cF_h^{int}} h_e^{-1}\|\b{\Lambda}_k^{z_0}- \Phi_h(\b{\Lambda}_k^{z_0})\|_{\b{L}^1(e)} \\ & \hspace{0.5cm}  + \eta_3  \sum_{e \in \cF_h^N} h_e^{-1}\|\b{\Lambda}_k^{z_0}- \Phi_h(\b{\Lambda}_k^{z_0})\|_{\b{L}^1(e)} + \eta_4 \sum_{e \in \cF_h^{C}} h_e^{-1}\|\b{\Lambda}_k^{z_0}- \Phi_h(\b{\Lambda}_k^{z_0})\|_{\b{L}^1(e)}  \\ & \hspace{0.5cm} + \sum\limits_{T\in \mathcal{T}_h} \int\limits_{T} \bm{\Xi}(\b{u}^{enr}_h-\b{u}_h): \bm{\epsilon}(\b{\Lambda}_k^{z_0}) ~dx  + \eta_5  \sum_{e \in \cF^{0}_h} \|{\bm{\Xi}}(\Phi_h(\b{\Lambda}_k^{z_0})\|_{\b{L}^1(e)} \\ & \hspace{0.5cm} + \eta_5  \sum_{e \in \cF^{0}_h}  h_e^{-1}\|\b{\Lambda}_k^{z_0}- \Phi_h(\b{\Lambda}_k^{z_0})\|_{\b{L}^1(e)} \\ & \lesssim \bigg( \sum_{i=1}^{4} \eta_i \bigg) \bigg(\sum_{T\in \mathcal{T}_h} h_T^{-2}\|\b{\Lambda}_k^{z_0}- \Phi_h(\b{\Lambda}_k^{z_0})\|_{\b{L}^1(T)} + \sum_{T\in \mathcal{T}_h} h_T^{-1}|\b{\Lambda}_k^{z_0}- \Phi_h(\b{\Lambda}_k^{z_0})|_{1,1,T}  \bigg)   \\ & \hspace{0.5cm} + \bigg(\max_{T\in \mathcal{T}_h} |\b{u}^{enr}_h-\b{u}_h|_{1,1,T} \bigg)  \sum_{T\in \mathcal{T}_h} |\b{\Lambda}_k^{z_0}|_{1,1,T} + \eta_5  \sum_{e \in \mathcal{F}^{0}_h} h_e^{-1} \|\b{\Lambda}_k^{z_0}- \Phi_h(\b{\Lambda}_k^{z_0})\|_{\b{L}^1(e)}  \\ & \hspace{0.5cm}+ \eta_5 \bigg( \sum_{ T \in \cT_h} h_T^{-1} |\Phi_h(\b{\Lambda}_k^{z_0})|_{1,1,T} + \sum_{ T \in \cT_h} |\Phi_h(\b{\Lambda}_k^{z_0})|_{2,1,T} \bigg) \\ & \lesssim  \bigg( \sum_{i=1}^{5} \eta_i \bigg) \bigg(\sum_{T\in \mathcal{T}_h} h_T^{-1}|\b{\Lambda}_k^{z_0}|_{1,1,T} \bigg)  \\ & \lesssim  (1+  |ln(h_{min})|^2) \bigg( \sum_{i=1}^{5} \eta_i \bigg).
\end{align*}
Using the Lemma \ref{lem4.5} and \eqref{ww}, we deduce the following estimate
\begin{align} \label{error}
\|\bm{u}-\bm{u}_h\|_{\bm{L}^{\infty}(\Omega)} \lesssim \mathcal{E}_h.
\end{align}
Next,  we establish an upper bound on $\langle \bm{R}_h, \bm{w} \rangle_{-1,1}$ for any $\bm{w} \in \bm{\mathcal{N}}$,  which is subsequently utilized for estimating error in the density forces.
\begin{lemma} \label{lem:Ghest}
	The following holds
	\begin{align}
	\langle \bm{R}_h, \bm{w} \rangle_{-1,1}  \lesssim  \bigg( \sum_{i=1}^{5} \eta_i \bigg) |\b{w}|_{\b{W}^{2,1}(\O)} \quad\forall ~\bm{w} \in \bm{\mathcal{N}}.
	\end{align} 
\end{lemma}
\begin{proof}
	Let $\bm{w} \in \bm{\mathcal{N}}$.  Now,  proceeding on the similar lines as in the subsection \ref{4.3} to deal  with the term $\langle \bm{R}_h , \b{w} \rangle_{-1,1}$ we find
	\begin{align*}
	\langle \bm{R}_h, \bm{w} \rangle_{-1,1} & \lesssim  \bigg( \sum_{i=1}^{5} \eta_i \bigg) \bigg(\sum_{T\in \mathcal{T}_h} h_T^{-1}|\b{w}|_{1,1,T} \bigg).
	\end{align*}
	{In the view of the Poincar$\acute{e}$'s type inequality $\|\b{\nabla} \psi\|_{L^2(\Omega)} \lesssim \|D^2\psi\|_{L^1(\Omega)}$ for $\psi \in W^{2,1}(\Omega) \cap H^1_0(\Omega)$  \cite[page 522]{nochetto2006pointwise}} and Cauchy-Schwarz inequality gives the following
	\begin{align*}
	\langle \bm{R}_h, \bm{w} \rangle_{-1,1}  \lesssim  \bigg( \sum_{i=1}^{5} \eta_i \bigg) \Big(\sum_{j=1}^{2} \sum_{T \in \cT_h}  |w_j|_{W^{1,2}(T)}\Big) \lesssim \bigg( \sum_{i=1}^{5} \eta_i \bigg) \Big(\sum_{j=1}^{2} \sum_{T \in \cT_h} |w_j|_{W^{2,1}(T)}\Big).
	\end{align*}
	Thus, we get the desired result.
\end{proof}
\par
\noindent
{
Finally, for $\bm{w} \in \bm{\mathcal{N}}$,   using Remarks \ref{rem1},  \ref{rem121},  equation \eqref{Gal1}  together with integration by parts and Lemma \ref{lem:Ghest} yields
\begin{align*}
 \langle {\bm{\sigma}}- {\bm{\sigma}}_h , \gamma_c(\b{w}) \rangle_c &= \langle {\bm{\sigma}},\gamma_c(\b{w}) \rangle_c -
 \langle {\bm{\sigma}}_h , \gamma_c(\b{w}) \rangle_c \\&= \tilde{\bm{\sigma}}(\bm{w}) -  \widehat{\bm{\sigma}}_h( \gamma_0(\bm{w}
 ))\\&=-a(\b{u}-\b{u}^{enr}_h,\b{w})+ \langle \bm{R}_h , \b{w} \rangle_{-1,1}\\ & \lesssim \Big(\|\b{u}-\b{u}^{enr}_h\|_{\bm{L}^{\infty}(\Omega)} + \sum_{i=1}^{5} \eta_i \Big)|\bm{w}|_{2,1,{ \Omega}},
\end{align*}
Therefore, 
\begin{align}
\|{\bm{\sigma}}-\bm{\sigma}_h\|_{-2,\infty, \Gamma_C} \lesssim \|\b{u}-\b{u}^{enr}_h\|_{\bm{L}^{\infty}(\Omega)} +\sum_{i=1}^{5} \eta_i.
\end{align}
Using relation \eqref{error} and Lemma
\ref{enriching_map},  we deduce
\begin{align} \label{error1}
\|{\bm{\sigma}}-\bm{\sigma}_h\|_{-2,\infty, \Gamma_C}\lesssim \mathcal{E}_h.
\end{align}}
\begin{proof}[\textbf{Proof of Theorem \ref{thm:rel11}}]
	Using the estimates \eqref{error} and \eqref{error1}, we conclude the proof of Theorem \ref{thm:rel11}.
	\end{proof}
\par
\noindent
Note that for the sake of notational convenience,  we denote $\eta_6$ to be $
		\|(E_{h,1}(\b{u}_{h})-\chi)^{+}\|_{L^{\infty}(\Gamma_C)}$ and $\eta_7$ as $ \|(\chi-E_{h,1}(\b{u}_{h}))^+\|_{L^{\infty}(\{\sigma_{h,1}>0\})}$.
\vspace{0.5 cm}
\par
\noindent
Next we aim to establish the lower bounds for the proposed error estimator $\mathcal{E}_h$.  
\subsection{Efficiency of the error estimator}
In this subsection, we discuss the local efficiency estimates for a posteriori error control for the quadratic DG FEM.  We refer to the article \cite{nochetto1995pointwise} for local efficiency estimates for pointwise a posteriori error analysis for linear elliptic problems. We begin by defining the error measure locally as follows: for any $D \subset \O$, denote 
\begin{align*}
Err_D =  \|\bm{u}-\bm{u}_h\|_{\bm{L}^{\infty}(D)} + \|{\bm{\sigma}}-\bm{\sigma}_h\|_{-2,\infty,D_1},
\end{align*} 
where $D_1 = \bar{D} \cap \Gamma_C$.  

 \vspace{0.2 cm}
 \noindent
The next theorem ensures that the residual estimator is bounded by error measure.  Before stating it we will define the following high order oscillation terms
\begin{itemize}
\item For $T \in \mathcal{T}_h$ define $$ Osc(\b{f}, T):= h_T^2 \|\b{f} -\bar{\b{f}}\|_{{\b{L}^{\infty}(T)}}.$$
\item For $e \in \mathcal{F}^N_h$ define $$ Osc(\b{\pi}, e):= h_e \|\b{\pi} -\bar{\b{\pi}}\|_{{\b{L}^{\infty}(e)}}.$$
\end{itemize}
Therein,  $\bar{\b{f}}$ and $\bar{\b{\pi}}$ denotes the piecewise constant approximations of the given data $\b{f}$ and $\b{\pi}$, respectively.
\begin{theorem}
Let $\b{u}$ and $\b{u_h}$ be the solution of equation \eqref{eq:CVI} and \eqref{eq:DVI},  respectively.  Then, the following hold
\begin{itemize}
\item $ h_T^2 \|\b{f} + \b{div}\b{\Xi}(\b{u}_h))\|_{\b{L}^{\infty}(T)} \lesssim Err_{T} + Osc(\b{f}, T) \quad\forall~T \in \mathcal{T}_h$,
\vspace{0.4 cm}
\item $ \|\sjump{\b{u}_h}\|_{{\b{L}^{\infty}(e)}} \lesssim \|\b{u}-\b{u}_h\|_{\b{L}^{\infty}(\omega_e)} \quad\forall~e \in \mathcal{F}^{0}_h$,
\vspace{0.4 cm}
\item $ h_e \|\sjump{\bm{\Xi}(\bm{u}_h)}\|_{{\b{L}^{\infty}(e)}} \lesssim Err_{{\mathcal{J}}} +Osc(\b{f}, T) \quad\forall~e\in \mathcal{F}^{0}_h~{and}~e \in \partial T, ~T \in \mathcal{T}_h$,
\vspace{0.4 cm}
\item $ h_e \|\b{\Xi}(\b{u}_h)\b{n}_e+ \b{\sigma}_h\|_{{\b{L}^{\infty}(e)}} \lesssim Err_{{\mathcal{J}}} + Osc(\b{f}, T) \quad\forall~e\in \mathcal{F}^C_h~{and}~e \in \partial T, ~T \in \mathcal{T}_h$,
\vspace{0.4 cm}
\item $ h_e \|\b{\pi}-\b{\Xi}(\b{u}_h)\b{n}_e\|_{{\b{L}^{\infty}(e)}} \lesssim Err_{{\mathcal{J}}} + Osc(\b{f}, T)+ Osc(\b{\pi}, e)\quad\forall~e\in \mathcal{F}^N_h~{and}~e \in \partial T, ~T \in \mathcal{T}_h$,
\end{itemize}
where $\omega_e$ refers to the set of triangles sharing the edge $e$.
\end{theorem}
\noindent
The following estimates can be proven using conventional bubble function techniques,  similar to those described in \cite{KP:2021:Signorini}.
\vspace{0.3 cm}
\par
\noindent
\begin{remark}
We remark here that although the efficiency of the estimators $\eta_6$ and $\eta_7$ has not been established theoretically,  it has been taken into account during the execution of the numerical experiments. 
\end{remark}
	\section{Numerical Experiments}\label{sec6}
	\noindent
The aim of this section is to demonstrate the numerical experiments that assess the performance of devised residual estimator $\cE_h$ defined in \eqref{total1}.  All the reported numerical results were carried in MATLAB\_R2020B. To verify the optimal behavior of the residual estimator $\cE_h$  at certain refinement level,  we first calculate the discrete solution $\b{u}_h$ and formulate  a posteriori error estimator $\cE_h$.   Using this information, we then compute the experimental order of convergence.  A key objective of a posteriori error estimation is to identify the elements that make significant contributions to the error and thus refine those triangles locally and repeat.  
\vspace{3 mm}
\par
\noindent
In the numerical simulation,  we start with an initial mesh $\mathcal{T}_{h,0}$ consisting of four congruent right angle triangles.  The adaptive mesh refinement strategy generates a sequence of meshes $\mathcal{T}_{h,k}, ~ k \in \mathbb{N}$ where at each iteration we consider the successive four modules:
\begin{align*}
\textbf{SOLVE} \longrightarrow \textbf{ESTIMATE} \longrightarrow \textbf{MARK}  \longrightarrow \textbf{REFINE} 
\end{align*}
In the module $\textbf{SOLVE}$ the discrete solution of variational inequality \eqref{eq:DVI} is computed using primal dual active set strategy \cite{hueber2005priori} on the mesh $\mathcal{T}_{h,k}$.  Followed by that,  in the module $\textbf{ESTIMATE}$,  the a posteriori error estimator $\cE_h$ is computed on each element of triangulation $\mathcal{T}_{h,k}$.  Further,  the module $\textbf{MARK}$ returns the set of marked elements which needs to be refined using maximum norm criterion \cite{verfurth1996review}.  Lastly in the module $\textbf{REFINE}$ we exploited newest vertex bisection algorithm \cite{Funken2011EfficientIO} to refine the marked triangles.

%
%

\begin{figure}[ht!] 
	\begin{center}
		\includegraphics[height=4cm,width=8cm]{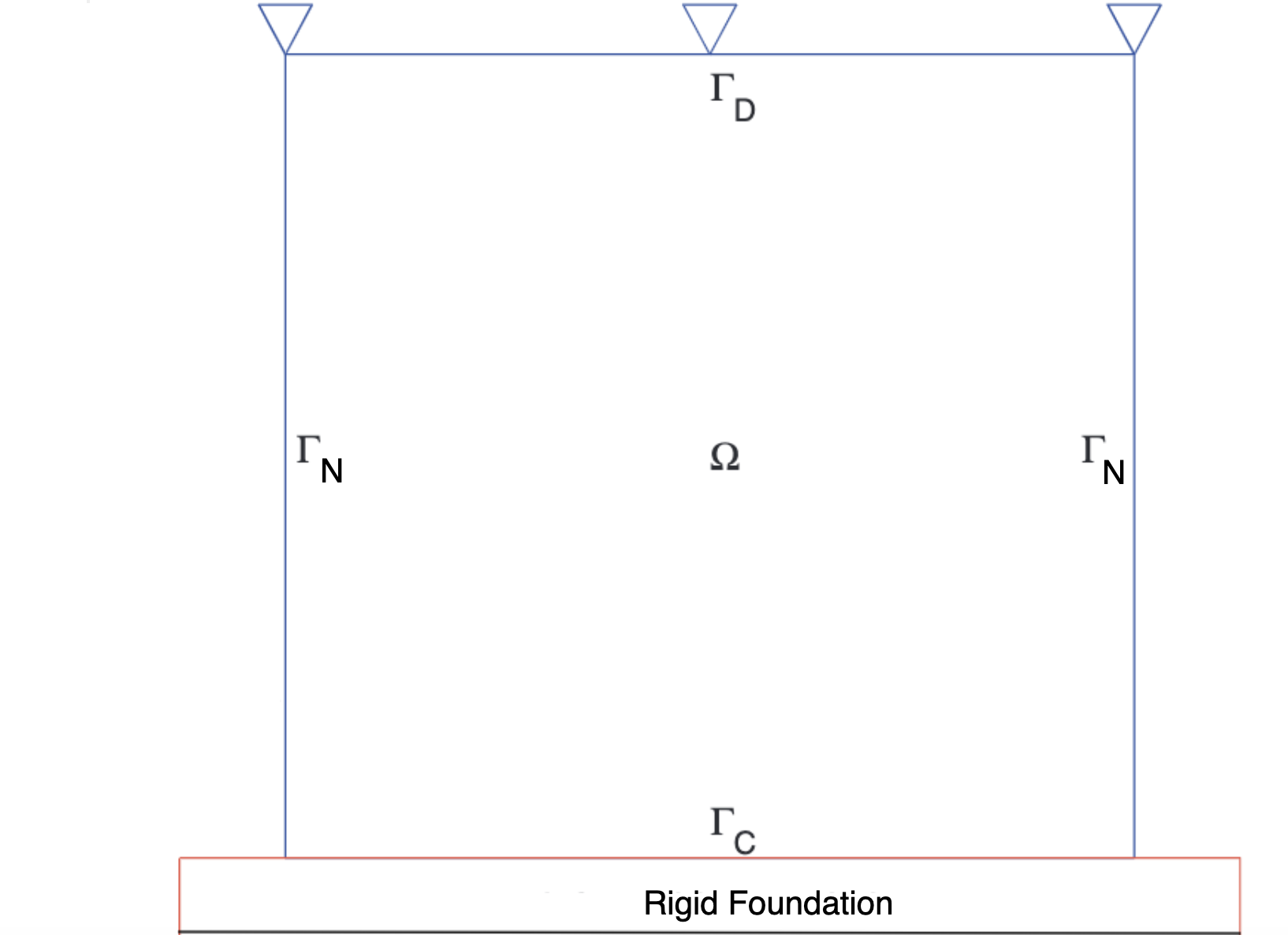}
		\caption{Physical setting of Model Problem 1.}
		\label{FIg1}
	\end{center}
\end{figure}

\begin{figure}
	\begin{subfigure}[b]{0.45\textwidth}
		\includegraphics[width=\linewidth]
		{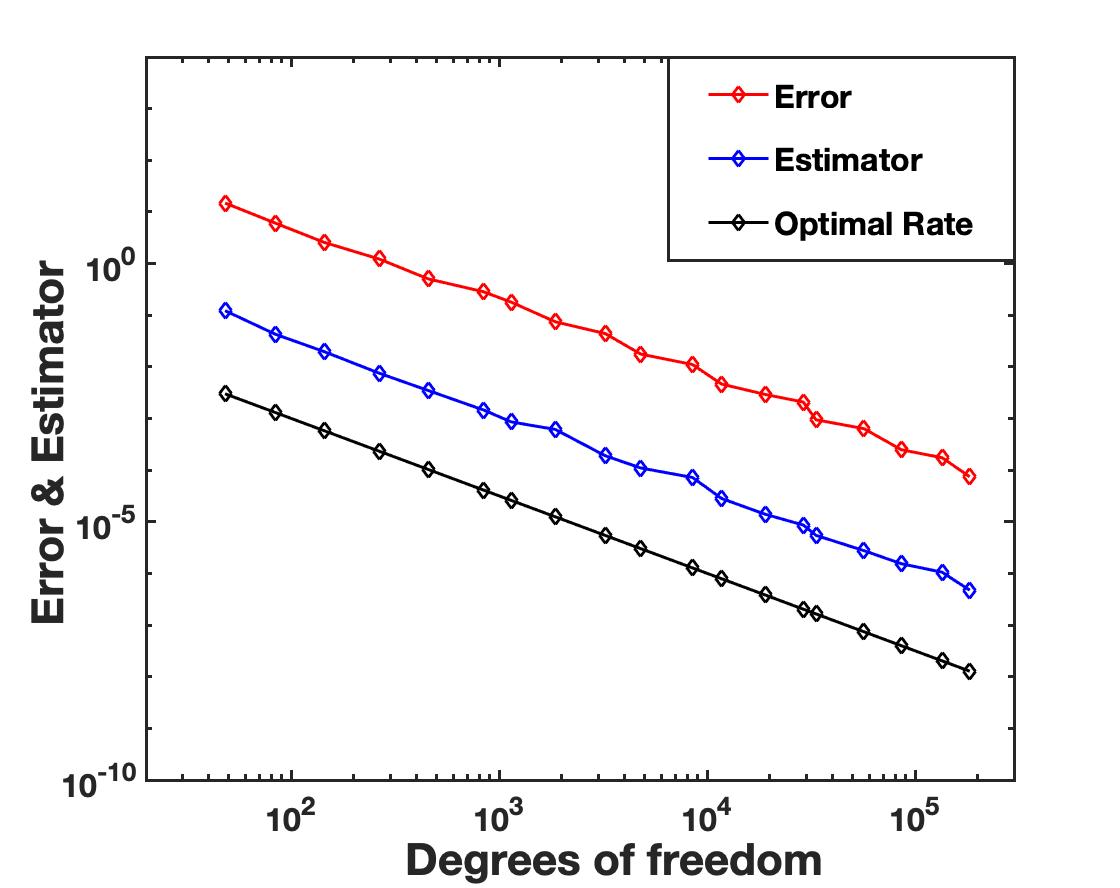}
		\caption{SIPG}
	\end{subfigure}
	\begin{subfigure}[b]{0.45\textwidth}
		\includegraphics[width=\linewidth]{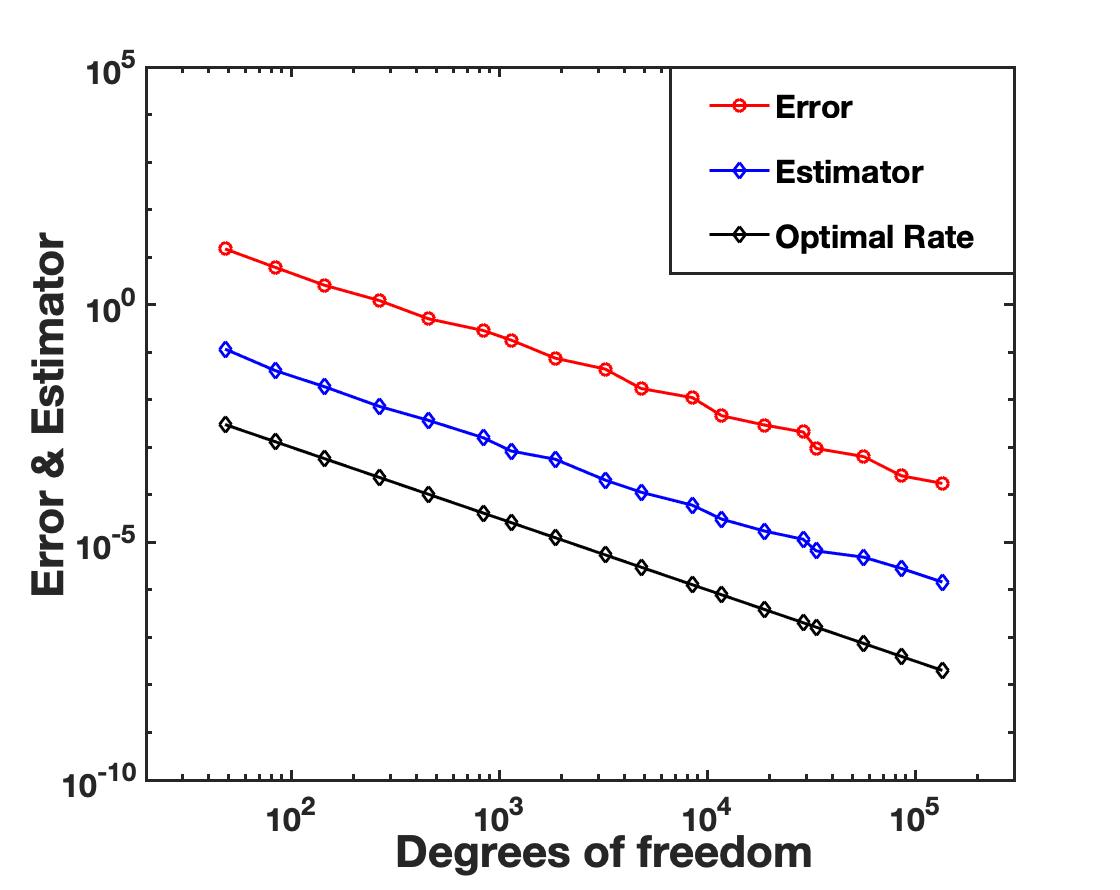}
		\caption{NIPG}
	\end{subfigure}
	\caption{Plot of convergence of Error and Estimator for SIPG and NIPG methods for  Model Problem 1. }\label{Example1_fullest}
\end{figure} 

\begin{figure}
	\begin{subfigure}[b]{0.45\textwidth}
		\includegraphics[width=\linewidth]
		{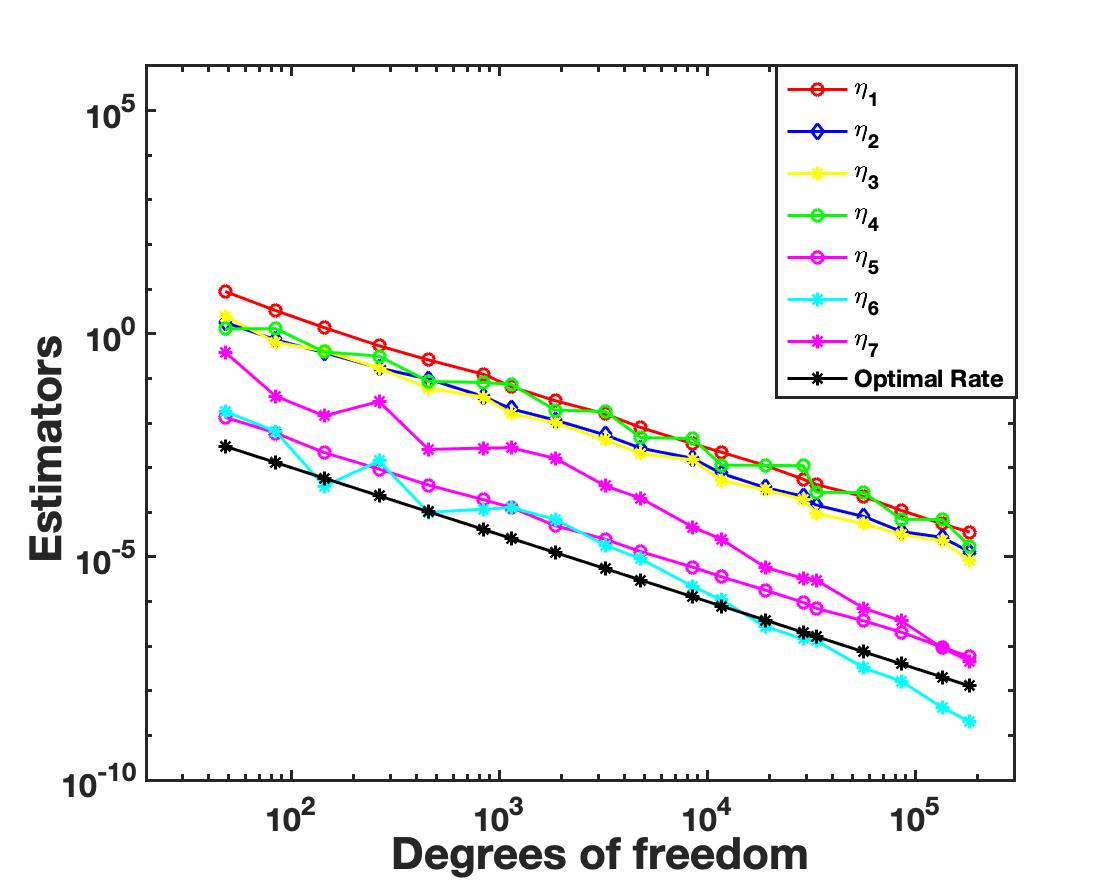}
		\caption{SIPG}
	\end{subfigure}
	\begin{subfigure}[b]{0.45\textwidth}
		\includegraphics[width=\linewidth]{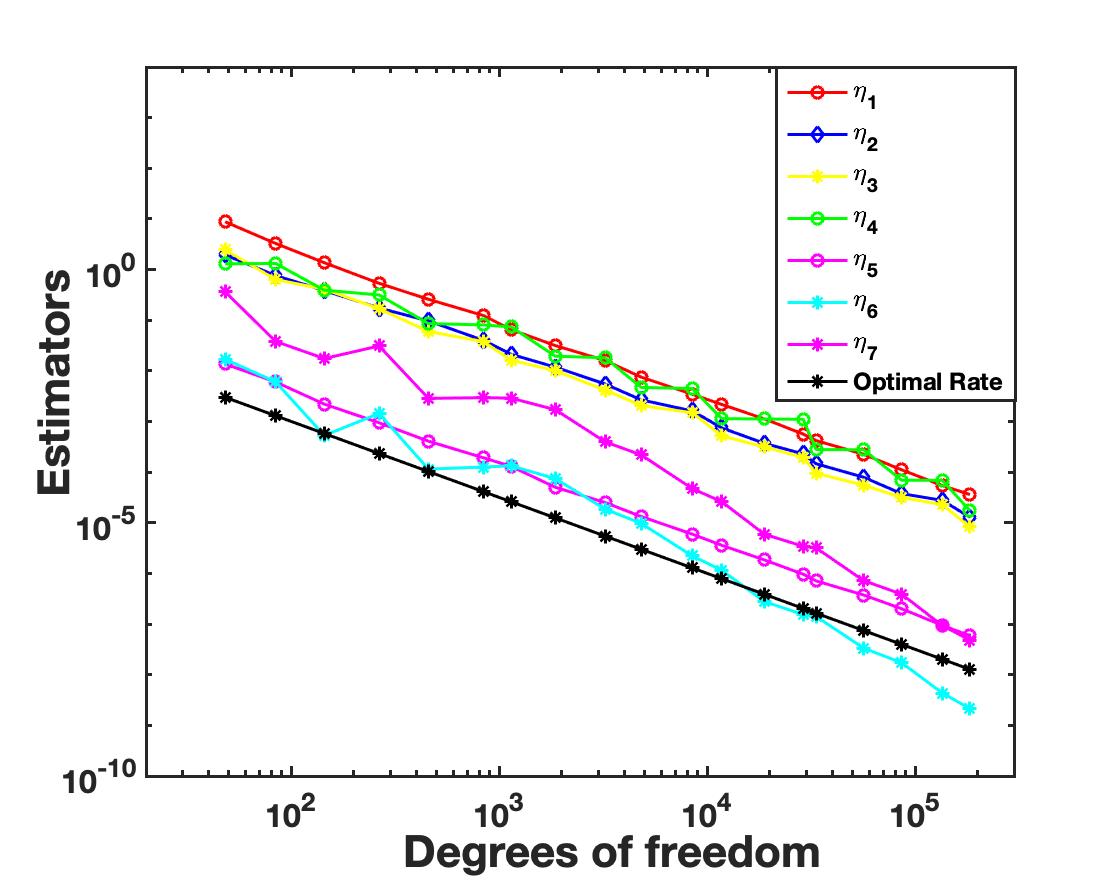}
		\caption{NIPG}
	\end{subfigure}
	\caption{Plot of residual estimators $\eta_i,  1\leq i \leq 7$ for SIPG and NIPG methods for Model Problem 1.}\label{Example1_indest}
\end{figure} 

\begin{figure}
	\begin{subfigure}[b]{0.45\textwidth}
		\includegraphics[width=\linewidth]
		{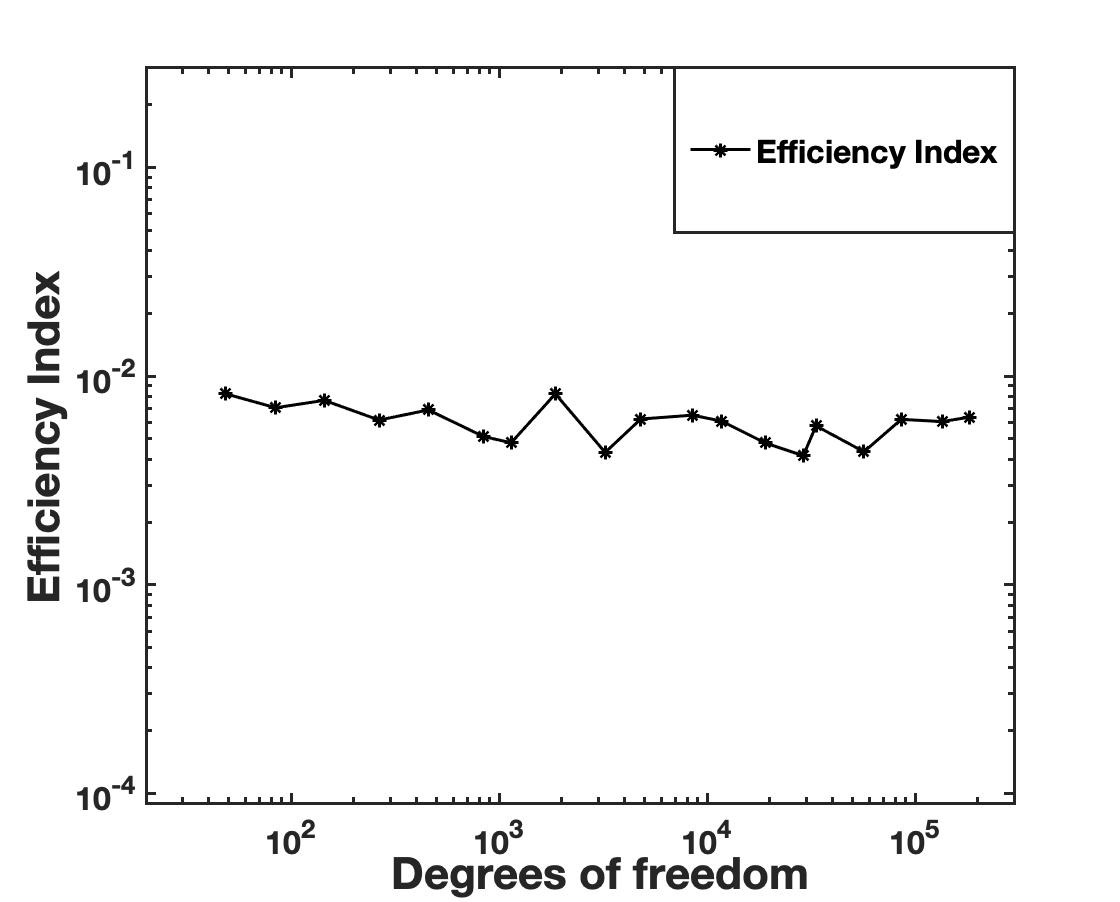}
		\caption{SIPG}
	\end{subfigure}
	\begin{subfigure}[b]{0.45\textwidth}
		\includegraphics[width=\linewidth]{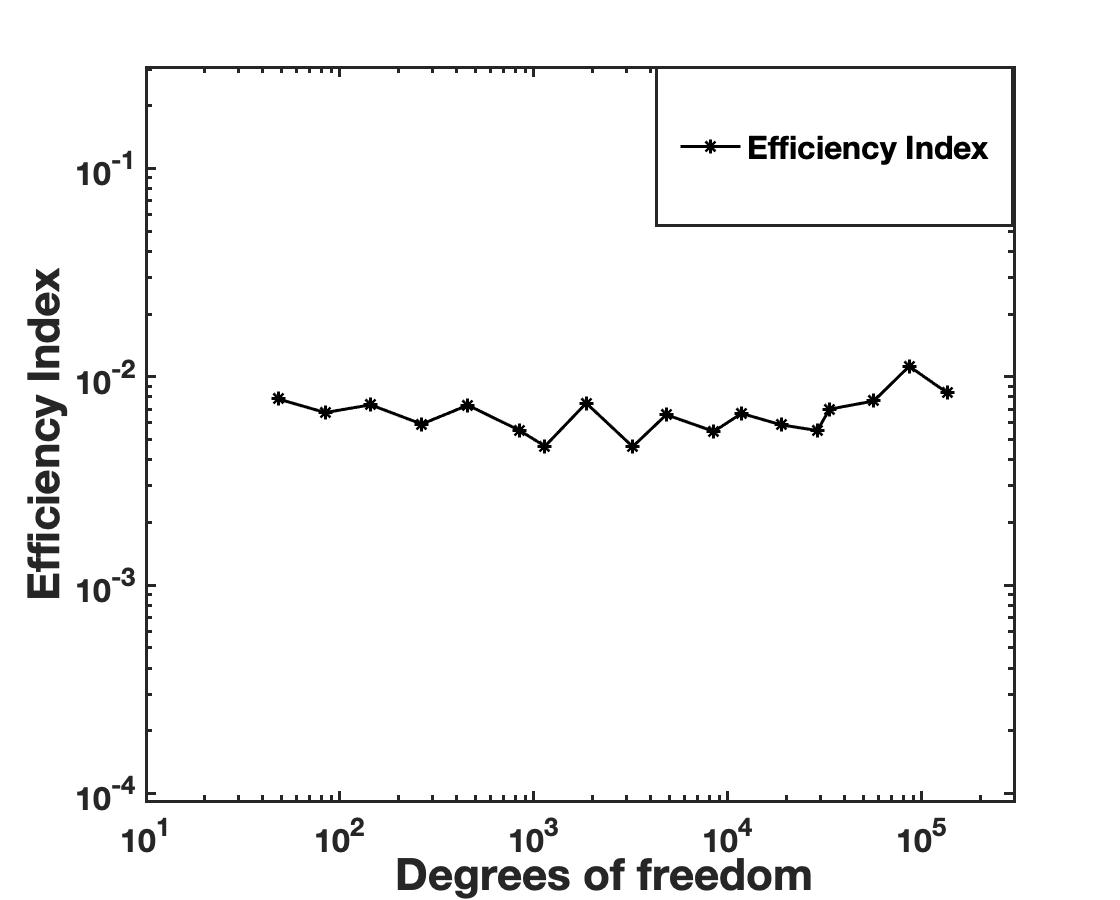}
		\caption{NIPG}
	\end{subfigure}
	\caption{Efficiency indices for SIPG and NIPG methods for Model Problem 1.}\label{Example1_Efficiency}
\end{figure} 
\par
\noindent

\vspace{0.3cm}
\par 
\noindent
\underline{\textbf{Model Problem 1}}\textbf{ \textit{(Contact with a rigid foundation)}}
\vspace{0.2 cm}
\\
Let us consider a unit square domain,  i.e. $\O=\{ (x,y): 0<x, y < 1\}$ with Dirichlet boundary as $y=1$.  Further, the Neumann boundary is assumed to be left and right side of square domain $\O$. The contact boundary is the bottom of square which is kept on the rigid foundation i.e. $\chi =0$.  We set rest of parameters as follows:
\begin{itemize}
\item The Lam$\acute{e}$ parameters $\mu$  and $\kappa$ are chosen to be 1.
\item The analytical solution to the problem is $\b{u} = (y^2(y-1),  (x-2)y(1-y)e^y).$
\item The source term $\b{f}$ and Neumann data $\b{\pi}$  
can be computed using the exact solution.
\item For the essential boundary condition,  we consider homogeneous Dirichlet boundary condition for $\b{u}$.
\end{itemize}
The pictorial representation of Model Problem 1 is described in Figure \ref{FIg1}. 
Further, we investigate the convergence property of residual estimator $\cE_h$.  Figure $\ref{Example1_fullest}$ depicts the optimal convergence of error and estimator for SIPG and NIPG methods for Model Problem 1.  The decay of each residual estimator $\eta_i, 1 \leq i \leq 7$ for SIPG and NIPG methods is shown in Figure \ref{Example1_indest}.  Clearly,  these observations are consistent with the theoretical findings.  In order to measure the quality of proposed error estimator,  we introduce the quantity \textit{Eff\_Index} which is defined as follows
\begin{align*}
\textit{Eff\_Index} := \dfrac {\textit{Estimator}}{\textit{Error}}.
\end{align*}
It is observed from Figure \ref {Example1_Efficiency} that the efficiency indices for both SIPG and NIPG methods are bounded above and below by generic constants,  thus ensuring the efficiency of residual error estimator $\cE_h$. 

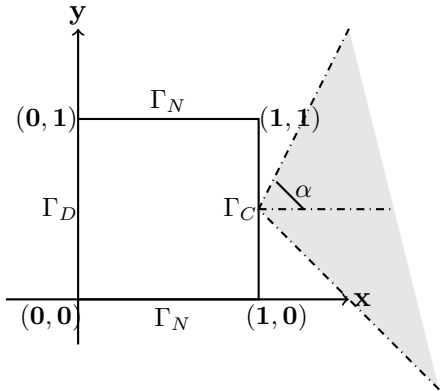
\begin{figure}[ht!]
\centering
\begin{tikzpicture}[scale=1.2]
\draw[thick] (0,0) --(2,0) --(2,2)--(0,2);
\draw[fill=black!10!white,dash dot,thick] (4,-1)--(2,1)--(3,3);
\draw[fill=black!10!white,dash dot,thick] (2,1)--(3.5,1);
\draw[thick,domain=293:360] plot({2+0.5*cos(\x)},{1.5-0.5*cos(\x)});
\draw[thick,->] (-0.8,0) -- (3,0);
\draw[thick,->] (0,-0.5) -- (0,3);
\node at (-0.3,-0.2) {\small$ \bf (0,0) $};
\node at (2.2,-0.2) {\small$ \bf (1,0) $};
\node at (2.35,2) {\small$ \bf (1,1) $};
\node at (-0.35,2.0) {\small$ \bf (0,1) $};
\node at (3.15,0) {$ \bf x $};
\node at (0,3.15) {$ \bf y $};
\node at (1,2.2) {\small$  \Gamma_N $};
\node at (-0.2,1) {\small$  \Gamma_D $};
\node at (1.8,1){\small$  \Gamma_C $};
\node at (1.05,-0.22) {\small$ \Gamma_N $};
\node at (2.5,1.2) {\small$ \alpha$};
\end{tikzpicture}
\caption{Physical setting of Model Problem 2.}
\label{fig3}
\end{figure}
\vspace{0.5 cm}
\par
\noindent
The next example corroborates the reliability and efficiency of a posteriori error estimator $\cE_h$ in the case of non-zero obstacle.  Therein,  the unit square comes in contact with a rigid wedge (see Figure \ref{fig3}).  Unlike the Model Problem 1,  we do not have exact solution in this case.
\vspace{0.2 cm}
\par
\noindent
\underline{\textbf{Model Problem 2}}\textbf{ \textit{(Contact with rigid wedge)}}
\\
\par
\noindent
In this example (adapted from \cite{walloth2019reliable}),  we simulate the deformation of unit elastic square $(0,1) \times (0,1)$  which come in contact with rigid wedge $\chi(y)=-0.2+0.5|y-0.5|$ inclined at an angle $\alpha= 63^{\circ}$ when displaced in $x$-direction.  We enforce the non-homogeneous boundary condition $\b{u}= (-0.1,0)$ on the Dirichlet boundary $\Gamma_D = \{0\} \times (0,1)$.   Rest,  the data for this problem is taken to be 
\begin{align*}
\Gamma_C &= \{1\} \times (0,1),~~
\Gamma_N =  ((0, 1)\times \{0\}) \cup ((0, 1)\times \{1\}),\\
 \b{f}&=(0,0)daN/mm^2,~~
\b{\pi}=(0, 0)daN/mm^2.
\end{align*}
Therein,  the values of Lam$\acute{e}$ parameters $\mu$ and $\kappa$ are calculated using the equations:
\begin{align*}
\mu = \frac{E}{2(1+\nu)}~~ \text{and}~~\kappa = \frac{E\nu}{(1+\nu)(1-2\nu)},
\end{align*}
where $E=500$ and $\nu=0.3$.
\begin{figure}
	\begin{subfigure}[b]{0.45\textwidth}
		\includegraphics[width=\linewidth]
		{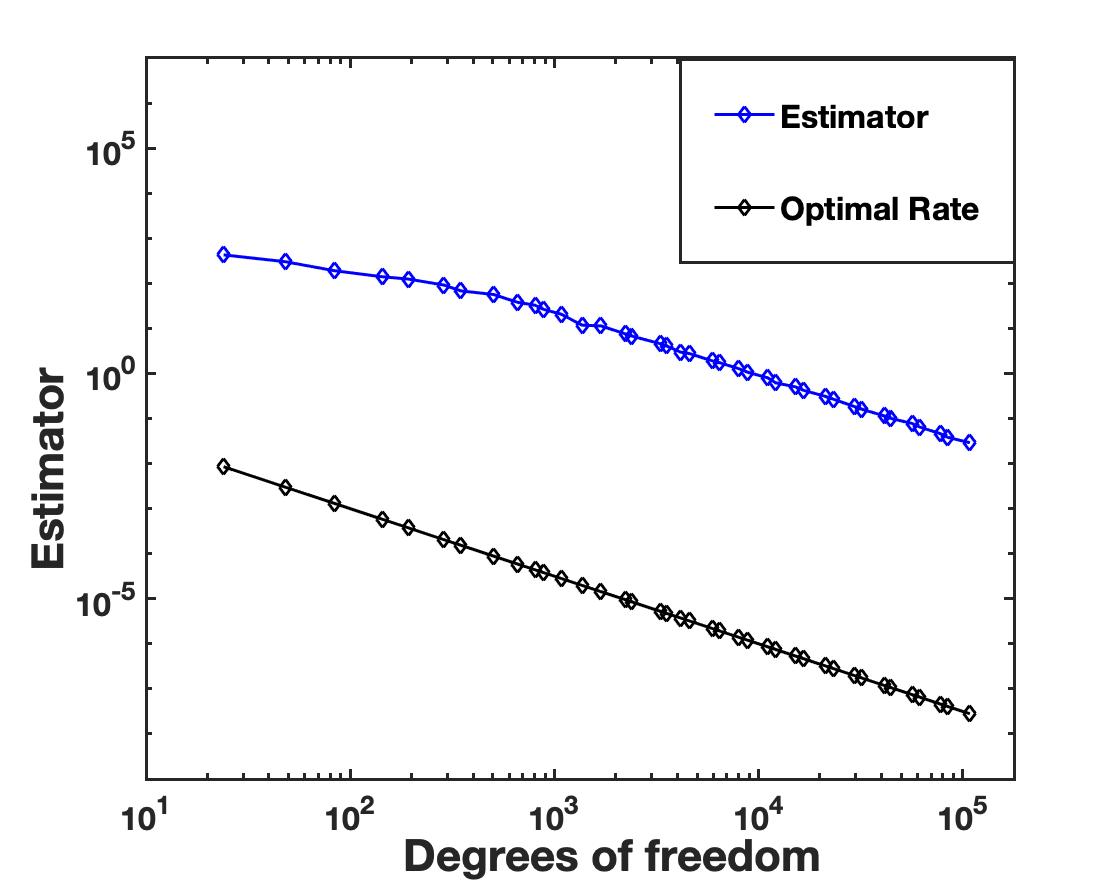}
		\caption{SIPG}
	\end{subfigure}
	\begin{subfigure}[b]{0.45\textwidth}
		\includegraphics[width=\linewidth]{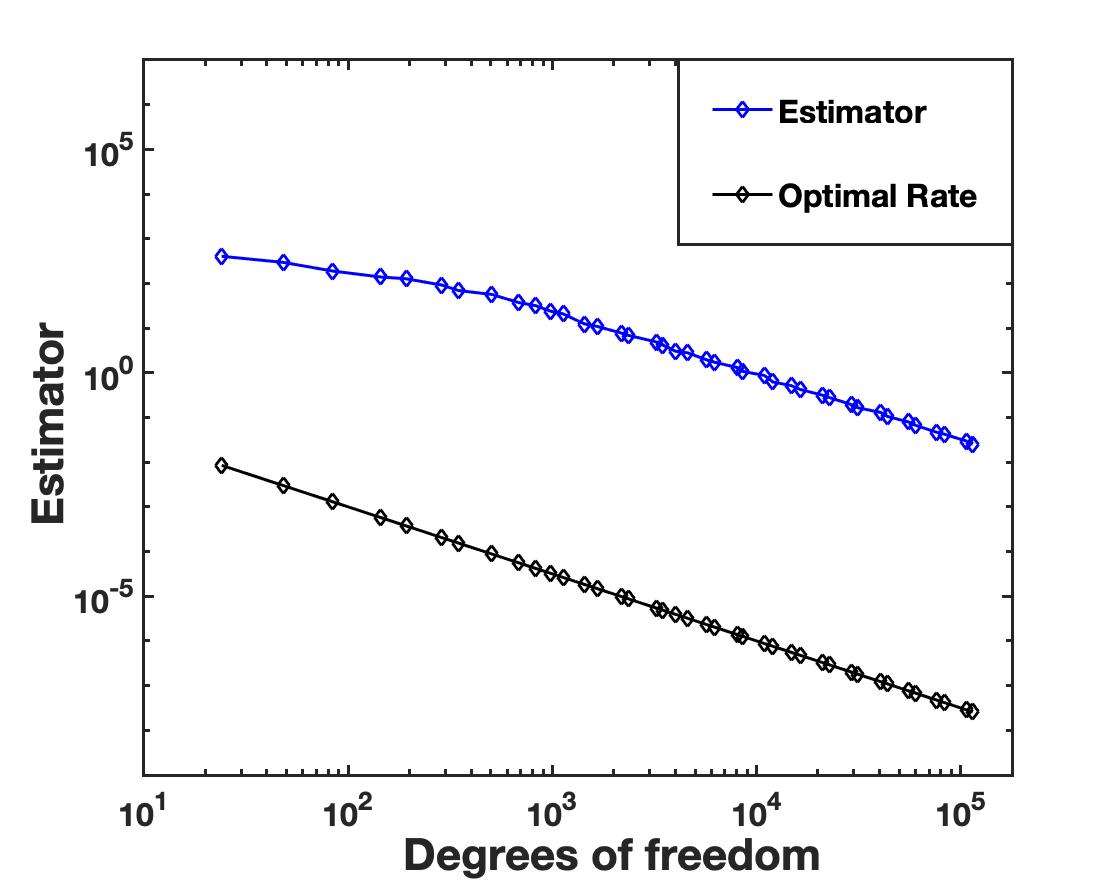}
		\caption{NIPG}
	\end{subfigure}
	\caption{Plot of convergence of Estimator for SIPG and NIPG methods for Model Problem 2.}\label{Example2_fullest}
\end{figure} 

\begin{figure}
	\begin{subfigure}[b]{0.45\textwidth}
		\includegraphics[width=\linewidth]
		{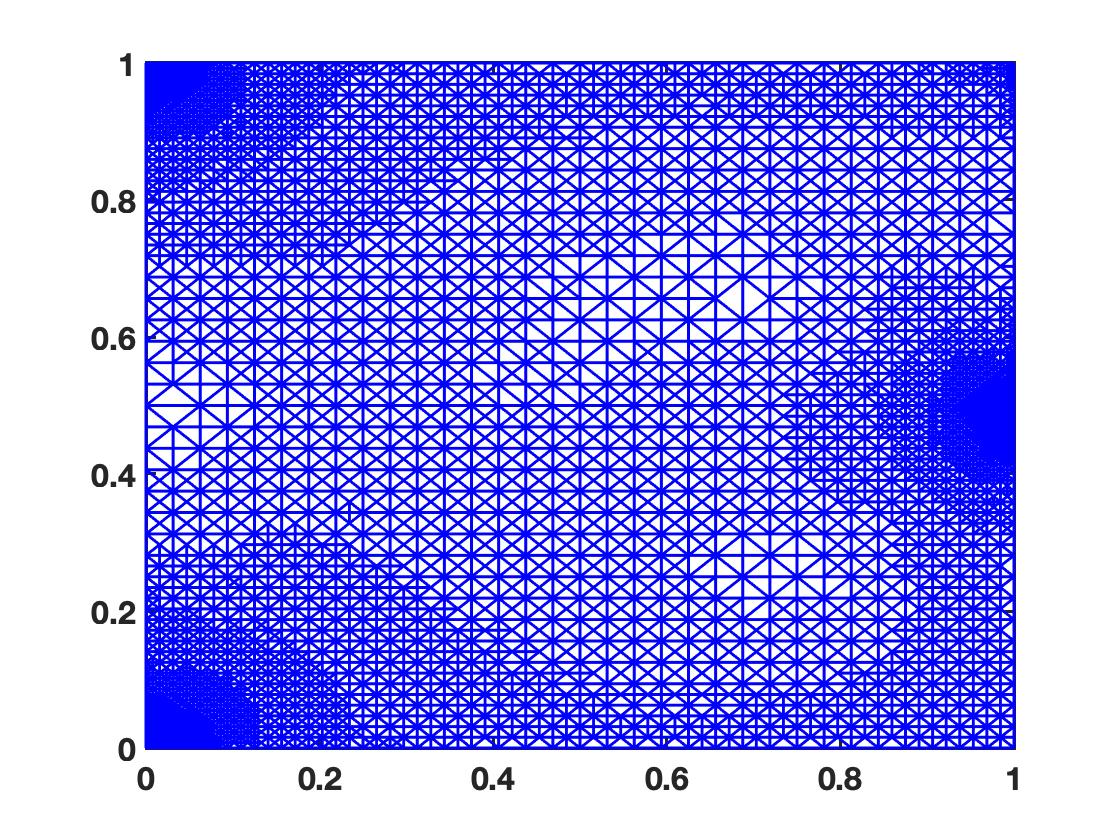}
		\caption{SIPG}
	\end{subfigure}
	\begin{subfigure}[b]{0.45\textwidth}
		\includegraphics[width=\linewidth]{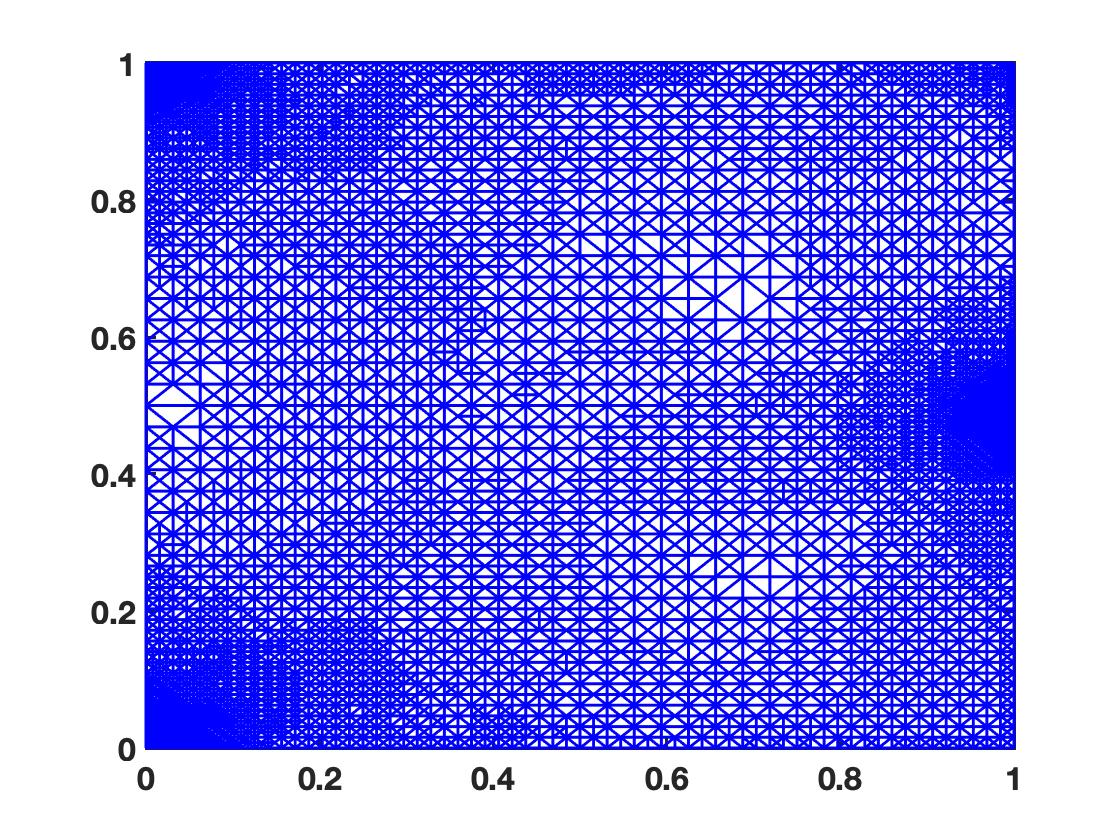}
		\caption{NIPG}
	\end{subfigure}
	\caption{Adaptive mesh refinement for SIPG and NIPG methods for Model Problem 2.}\label{Example2_mesh}
\end{figure} 
\vspace{0.2cm}
\par 
\noindent
In Figure \ref{Example2_fullest} we report the optimal convergence of estimator versus the degrees of freedom for SIPG and NIPG methods, respectively.  
Figure \ref{Example2_mesh} illustrates the adaptive mesh at an intermediate level generated by the proposed adaptive algorithm for both the methods.  It is observed that the refinement mainly concentrates around the region where tip of wedge intends on the contact boundary and the intersection of Dirichlet and Neumann boundary.  Thus,  the residual error estimator $\cE_h$ well captures the singular behavior of the solution.  The reduction of individual estimators $\eta_i ~( i = 1 : 7)$ with the increasing degrees of freedom for both the methods  is shown in Figure \ref{Example2_Indest}.

\begin{figure}
	\begin{subfigure}[b]{0.45\textwidth}
		\includegraphics[width=\linewidth]
		{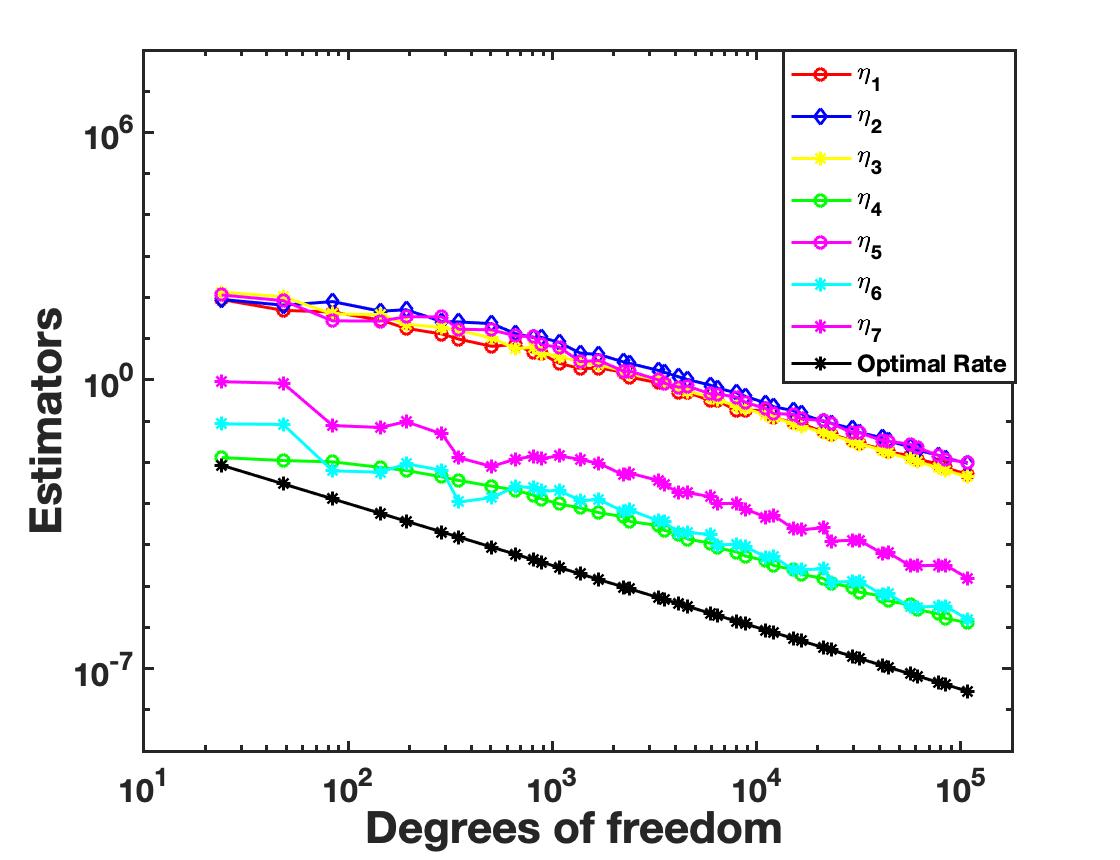}
		\caption{SIPG}
	\end{subfigure}
	\begin{subfigure}[b]{0.45\textwidth}
		\includegraphics[width=\linewidth]{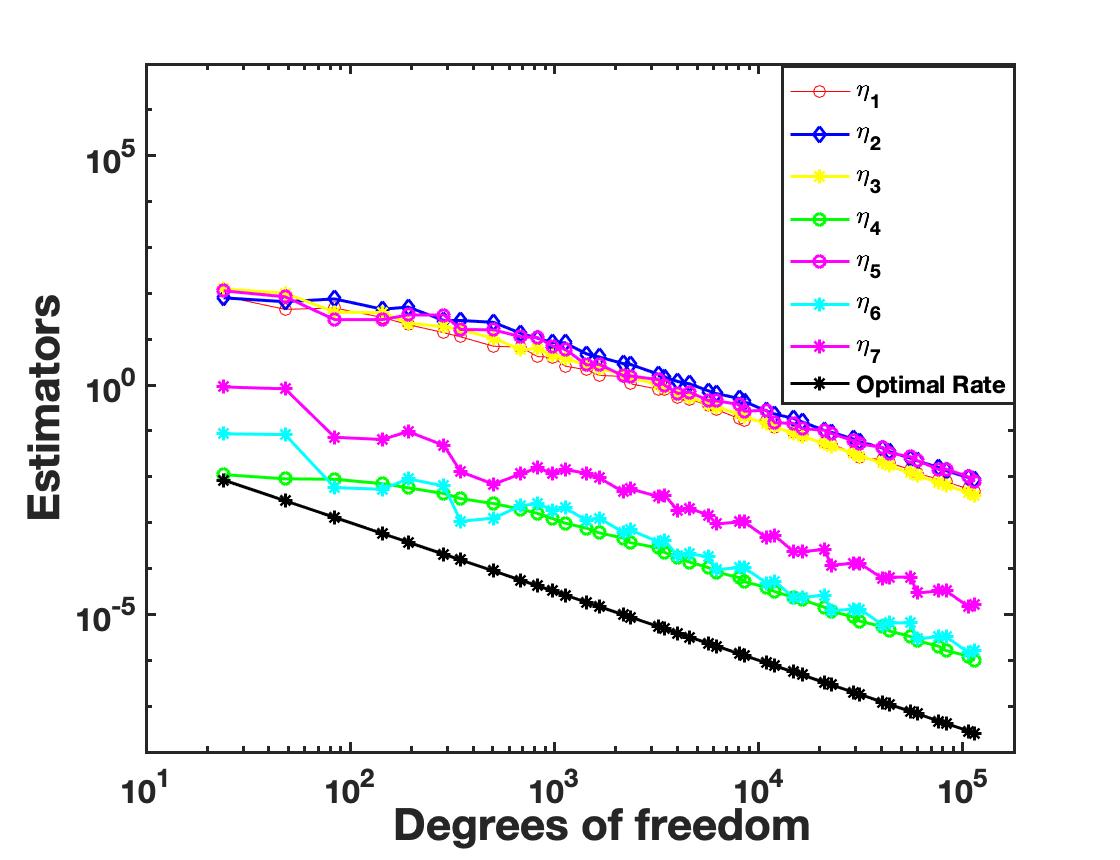}
		\caption{NIPG}
	\end{subfigure}
	\caption{Plot of individual estimators $\eta_i, 1 \leq i \leq 7$ of SIPG and NIPG methods for Model Problem 2.}\label{Example2_Indest}
\end{figure}

\vspace{0.5cm}
\par
\noindent

	\bibliographystyle{unsrt}
	\bibliography{rohi_4_3_2022}
\end{document}